%% file: NSRevision.tex
\documentclass[10pt]{article}

\usepackage{amsfonts,color}
\usepackage{amsmath,amssymb}
\usepackage{epsfig}
\usepackage{lscape}
\usepackage{amssymb}
\usepackage{graphicx}
\usepackage[utf8]{inputenc}	
\usepackage{amsthm,amssymb}
\usepackage{amsfonts}
\usepackage[english]{babel}
\usepackage{dsfont}
\usepackage{bbm}
\usepackage[T1]{fontenc}
\usepackage{colonequals}
\usepackage{geometry}
\usepackage{mathtools}
\usepackage{csquotes}
\usepackage{url}
\usepackage{lipsum}
\usepackage{array}
\usepackage[utf8]{inputenc}

\usepackage{float}
\usepackage{fancybox}
\usepackage{calc}
\usepackage{mathrsfs}
\usepackage{amsmath}
\usepackage{amsthm}
\usepackage{amssymb}
\usepackage{graphicx}
\usepackage{wasysym}
\usepackage{etoolbox}

\newcommand{\id}{{\boldsymbol{\mathbbm{1}}}}

\input{neffstyle}
\allowdisplaybreaks[1]

\makeindex

\newtheorem{theorem}{Theorem}[section]

\newtheorem{lemma}[theorem]{Lemma}

	\def\tr{\textrm{tr}}
	
	\def\dvg{\textrm{Div}}
	\def\crl{\textrm{Curl}}

	\def\dd{\displaystyle}

\def\C{\mathbb{C}}

\def\L{\mathbb{L}_c}

\newtheorem{remark}{Remark}[section]
\makeatletter
\let\@fnsymbol\@arabic
\makeatother

\begin{document}
	\global\long\def\u{u}
	\global\long\def\p{P}
	\global\long\def\X{X}
	\global\long\def\me{\mu_{e}}
	\global\long\def\sym{\textrm{sym}}
	\global\long\def\grad{\nabla}
	\global\long\def\le{\lambda_{e}}
	\global\long\def\tr{\textrm{tr}}
	\global\long\def\mc{\mu_{c}}
	\global\long\def\skew{\textrm{skew}}
	\global\long\def\curl{\textrm{Curl}}
	\global\long\def\ac{\alpha_{c}}
	\global\long\def\mh{\mu_{\mathrm{mic}}}
	\global\long\def\lh{\lambda_{\textrm{mic}}}
	\global\long\def\B{\mathscr{B}}
	\global\long\def\R{\mathbb{R}}
	\global\long\def\fr{\rightarrow}
	\global\long\def\Q{\mathcal{Q}}
	\global\long\def\A{\mathscr{A}}
	\global\long\def\L{\mathscr{L}}
	\global\long\def\D{\mathscr{D}}
	\foreignlanguage{italian}{}\global\long\def\id{\mathds{1}}
	\global\long\def\ds{\textrm{dev sym}}
	\global\long\def\sph{\textrm{sph}}
	\global\long\def\eg{\boldsymbol{\varepsilon}}
	\global\long\def\aa{\boldsymbol{\alpha}}
	\global\long\def\o{\boldsymbol{\omega}}
	\global\long\def\ege{\boldsymbol{\varepsilon}_{e}}
	\global\long\def\egp{\boldsymbol{\varepsilon}_{p}}
	\global\long\def\punto{\,.\,}
	\global\long\def\so{\mathfrak{so}}
	\global\long\def\Sym{\textrm{Sym}}
	\global\long\def\MM{\boldsymbol{\mathfrak{M}}}
	\global\long\def\C{\mathbb{C}}
	\global\long\def\gl{\mathfrak{gl}\left(3\right)}
	\global\long\def\P{P}
	\global\long\def\Lin{\textrm{Lin}}
	\global\long\def\D{\boldsymbol{D}}
	\global\long\def\a{\alpha}
	\global\long\def\b{\beta}
	\global\long\def\lle{\lambda_{e}}
	\global\long\def\ce{\mathbb{C}_{e}}
	\global\long\def\cm{\mathbb{C}_{\mathrm{micro}}}

	\global\long\def\axl{\textrm{axl}}
	\global\long\def\dev{\textrm{dev}}
	\global\long\def\Ls{\widehat{\mathbb{L}}}
	\global\long\def\mum{\mu_{\mathrm{macro}}}
	\global\long\def\lam{\lambda_{\mathrm{macro}}}
	\global\long\def\mh{\mu_{\mathrm{micro}}}
	\global\long\def\lh{\lambda_{\textrm{micro}}}
	\global\long\def\vau{\omega_{1}^{int}}
	\global\long\def\vad{\omega_{2}^{int}}
	\global\long\def\axl{\textrm{axl}}
	
	\title{Nonstandard micro-inertia terms in the relaxed micromorphic model: well-posedness for
		dynamics }
	\author{
		Sebastian Owczarek\,\thanks{Sebastian Owczarek, \ \  Faculty of Mathematics and Information Science, Warsaw University of Technology, ul. Koszykowa 75, 00-662 Warsaw, Poland; email: s.owczarek@mini.pw.edu.pl}\quad
		and \quad	Ionel-Dumitrel Ghiba\footnote{Ionel-Dumitrel Ghiba, \   Alexandru Ioan Cuza University of Ia\c si, Department of Mathematics,  Blvd. Carol I, no. 11, 700506 Ia\c si,
			Romania;  Octav Mayer Institute of Mathematics of the
			Romanian Academy, Ia\c si Branch,  700505 Ia\c si; email: dumitrel.ghiba@uaic.ro} 
		\quad
		and \quad Marco-Valerio d'Agostino\footnote{Marco-Valerio d'Agostino, \ \  Laboratoire de G\'{e}nie Civil et Ing\'{e}nierie Environnementale,
			Universit\'{e} de Lyon-INSA, B\^{a}timent Coulomb, 69621 Villeurbanne
			Cedex, France,
			email: marco-valerio.dagostino@insa-lyon.fr}	\quad
		and \quad
		\\	Patrizio Neff\,\thanks{Patrizio Neff, \ \  Head of Lehrstuhl f\"{u}r Nichtlineare Analysis und Modellierung, Fakult\"{a}t f\"{u}r Mathematik, Universit\"{a}t Duisburg-Essen, Campus Essen, Thea-Leymann Str. 9, 45127 Essen, Germany, email: patrizio.neff@uni-due.de}
		 }
	\maketitle
\begin{abstract}
We study the existence of  solutions  arising from the modelling of elastic materials using generalized theories of continua. In view of some evidence from physics of meta-materials we focus our effort on two recent nonstandard relaxed micromorphic models including novel micro-inertia terms. These novel micro-inertia terms are needed to better capture the band-gap response. The existence proof is based on the Banach fixed point theorem.

\bigskip

\smallskip

\end{abstract}
\tableofcontents
\section{Introduction}

\subsection{Preliminaries and motivation}
The micromorphic theory \cite{Eringen64,Mindlin64} is a generalised theory of continua which, in order to describe both macro- and micro-deformation, considers that any point of the body is endowed with two   {fields}, a vector field $u:\Omega\times[0,T]\rightarrow\mathbb{R}^3$ for the displacement of the macroscopic material points, and a tensor field 
 $P:\Omega\times[0,T]\rightarrow\mathbb{R}^{3\times3}$ describing the micro-deformation (micro-distortion) of  the substructure of  the material. This theory was introduced 58 years ago by Eringen and Mindlin and an important part of its reason to be was to have a model which is capable to improve  agreement between the analytical and numerical results regarding wave propagation indicated by the model and those obtained in experiments on actual materials. When  the classical theory of linear
 elasticity is used, no dispersive effects can be predicted. However, the elastic wave propagation through heterogeneous media is generally
 dispersive: each wave number travels with a distinct
wave speed.
 Some other possible applications of the micromorphic model  are described in \cite{Eringen99}.

 It is well-known that the Cosserat theory (the case when the micro-distortion  $P$ is antisymmetric) is a particular case of the micromorphic theory. The Cosserat theory was introduced in 1909  \cite{Cosserat09} (see also \cite{Neff_Chelminski05_dyn,Neff_JeongMMS08,Neff_curl06,Neff_micromorphic_rse_05,neff2015existence}) but it was not really taken into account until the micromorphic theory was introduced.   The biggest shortcoming of the classical micromorphic theory is that it involves an excessively large  number of constitutive coefficients, which must be determined if one wants to  use it in applications.  In order to avoid this shortcoming, one solution  is to adapt the micromorphic theory to the phenomena we want to model. This is done for instance in the works \cite{Neff_Forest07,Forest06} in which $P$ is assumed to be symmetric (the so called microstrain model).  
 
 In a recent paper \cite{NeffGhibaMicroModel}, we have introduced   a modified micromorphic theory which we  called the relaxed micromophic model.  This model completes the model which was proposed in \cite{Eringen_Claus69} and which was critically discussed in \cite{Teodosiu}. Since Claus and Eringen  have introduced such a model having in mind its application to the dislocation theory, the main criticism centered about  the non-symmetry of the force stress tensor. At that time it was not clear that such a theory may have a chance to be well-posed, even  when the  force stress tensor is symmetric. However, due to some new Poincar\'{e}-Korn type inequalities \cite{BNPS2, NeffPaulyWitsch}, the relaxed micromorphic model is well-posed  also when the  force stress tensor is symmetric \cite{GhibaNeffExistence}.
We have explained in \cite{NeffGhibaMadeoLazar} why the relaxed model is a particular case of the classical Mindlin-Eringen model and how the energy of the relaxed model may be obtained taking some suitable form of the constitutive coefficients in the classical theory. The relaxed micromorphic theory is still general enough to incorporate   the kinematics of the most used particular theories of the classical micromorphic theory, see  \cite{NeffGhibaMicroModel} and Figure \ref{mr}.
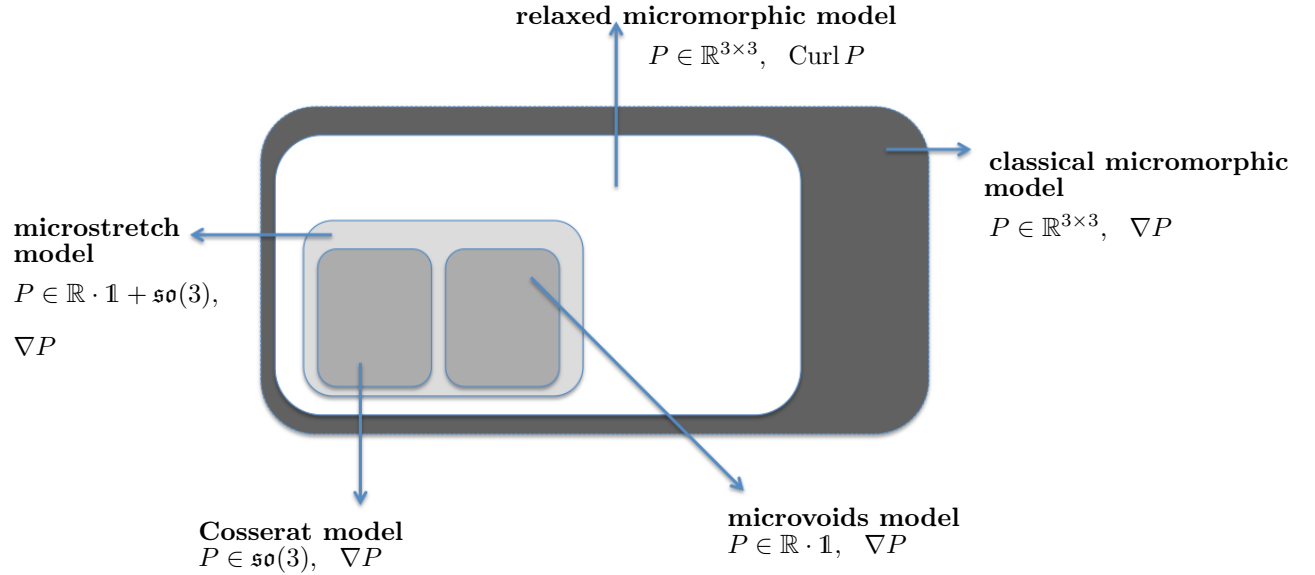
\begin{figure}[H]
	\begin{centering}
			\hspace{2cm}\includegraphics[scale=0.4]{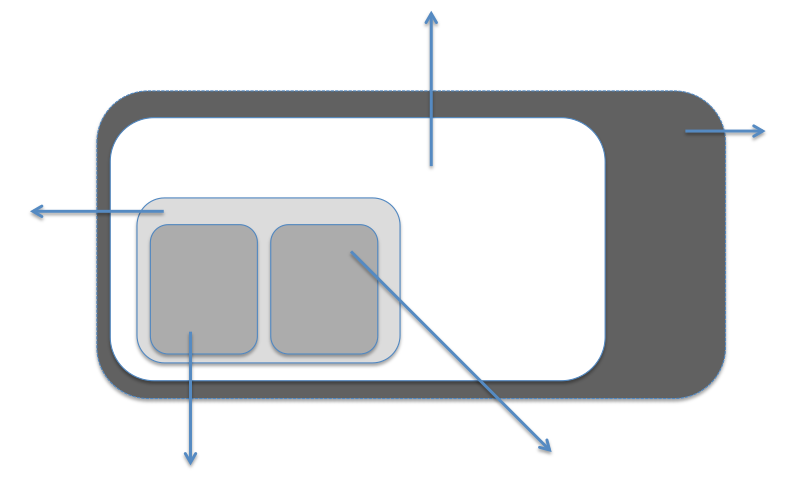} 
			\put(-1,135){\textbf{classical micromorphic}}
			\put(-3,125){\textbf{model}}
			\put(-1,110){$P\in \mathbb{R}^{3\times 3},\ \ \nabla P$}
		\put(-180,190){\textbf{relaxed micromorphic model}}
		\put(-130,175){$P\in \mathbb{R}^{3\times 3},\ \ {\rm Curl}\, P$}
	\put(-370,110){\textbf{microstretch}}
	\put(-370,85){$P\in \mathbb{R}\cdot \id+\so(3),$}
	\put(-370,65){$\nabla P$}
\put(-370,100){\textbf{model}}
\put(-300,-5){\textbf{Cosserat model}}
\put(-300,-15){$P\in \so(3),\ \ \nabla P$}
\put(-100,0){\textbf{microvoids model}}
\put(-100,-10){$P\in \mathbb{R}\cdot \id,\ \ \nabla P$}\end{centering}
	\caption{\label{mr} \footnotesize{An illustration of the relations between different theories of materials with microstructure: the classical micromorphic model \cite{Eringen64}; the linear  microstretch  model \cite{Eringen99}; the linear isotropic Cosserat model \cite{Cosserat09};  the linear isotropic microvoids model \cite{NunziatoCowin79}. Comparison with other models are discussed in \cite{NeffGhibaMicroModel}.} }
	
\end{figure}

Considering a wave ansatz in the partial differential equations arising in the relaxed micromorphic model, some band gaps arise in the diagrams describing the dispersion curves. This means that there exists an interval of wave frequencies for which no wave propagation may  occur \cite{MadeoNeffGhibaW,madeo2016complete,madeo2016reflection}. Such an interesting phenomena is  obtained  when meta-materials are designed and may not be captured by the classical micromorphic approach \cite{d2017panorama,madeo2017review}. These meta-materials are able to “absorb” or even “bend” elastic waves with no energetic cost (see e.g. \cite{blanco2000large,liu2000locally}). They are conceived arranging small components into periodic or quasi-periodic patterns in such a way that the resulting structure possesses new unimaginable properties with respect to the original material.
 This insight has opened a way  to obtain    new generalized continuum models allowing to describe the behavior of meta-structures in the simplified framework of continuum mechanics with homogenized material behavior. Moreover, due to the a priori relation between macro and micro parameters of the considered samples obtained in \cite{barbagallo2017transparent},  the constitutive  parameters of our model can be identified on real meta-materials \cite{Barbagalo2018,d2017effective} even for anisotropic materials, opening the way to the efficient design and realization of meta-structures.

Besides the interesting property of the relaxed micromorphic theory to model certain band-gap meta-materials, the mathematical results which can be obtained are also surprising:  it is noted that all the results which assume that the energy is positive definite in the relaxed model are in fact valid without assuming that the corresponding energy is positive definite in the classical Mindlin-Eringen model. Therefore, for the relaxed model we have obtained \cite{GhibaNeffExistence} extensions of the results established by S\'{o}os \cite{Soos69}, by Hlav{\'a}{\v c}ek \cite{Hlavacek69},  by Ie\c san and Nappa \cite{IesanNappa2001} and by Ie\c san \cite{Iesan2002} which all use the positive definitness in the Mindlin-Eringen model. We  mention  also the  results given  by Picard et al. \cite{picard} which have proved an existence result in the linear theory of micromorphic elastic solids using an interesting   mother/descendant mechanism regarding various models. On the other hand, almost all the mathematical results  are obtained in the classical Mindlin-Eringen model assuming  a strong anchoring boundary condition: the microdistortion $P$ is completely described at one part of the boundary. Our model gives  us the possibility to relax this condition and to be able to prescribe only the tangential components of the rows of $P$.

Going back to the applications of the relaxed micromorphic model in the modelling of wave propagation in meta-materials, it can be observed that for large wavenumbers the dispersion curves obtained from our model have a different behavior than  those which are predicted by the 
Bloch-Floquet analysis. By introducing some nonstandard micro-inertia terms in the form of the kinetic energy the fitting between the dispersion curves given by the Bloch-Floquet analysis and by the new model is considerably improved \cite{madeo2017role,madeo2018modeling,d2017effective,aivaliotis2018low}.    {In a recent paper \cite{Barbagalo2018} it is shown that while the linear elastic Cauchy  continuum is able to describe the pulse propagation through the
micro-structured domain only in the low frequency regime,  the relaxed micromorphic
model with these nonstandard micro-inertia terms is able to account for the overall behavior of the material for  frequency regimes
including the band-gap and higher frequencies. Moreover, this property  seems to be  difficult to obtain with other generalized continuum models.}

The idea of considering some new terms in the form of the kinetic energy is not entirely new, it already appeared in the linear  theory of nonlocal elasticity  introduced in \cite{eringen1972linear} and derived from a more general theory developed by Eringen and Edelen \cite{eringen1972nonlocal}. Gradient elasticity can also be used to capture
the dispersive behavior of waves that propagate through a heterogeneous material and  the presence of micro-inertia terms was used, before the work of Eringen \cite{eringen1972linear}, also by Mindlin  \cite[page 68]{Mindlin64}. The difference is that Mindlin used micro-inertia terms only in the context of gradients elasticity, while Eringen used them without considering higher order derivatives of the displacement \cite{askes2011increasing}. Various  formats of gradient elasticity with inertia gradients have been considered in many works \cite{rubin1995continuum,wang2002modeling,chen2001dispersive,domenico2016new,bennett2007elasticity,gitman2007gradient}. But there does not seem to  exist a rigorous mathematical treatment of these models.

The main aim of this paper is to present  relaxed micromorphic models  including some nonstandard micro-inertia terms and to prove that these models  are well-posed. Our analysis covers also the case when the stress tensor is symmetric since, in our mathematical approach, the presence of the constitutive term which makes it non-symmetric is not essential: for the mathematical analysis of the model the presence of the Cosserat couple modulus  is redundant. However, in the study of meta-materials this coefficient is essential when one wants to model the band-gap phenomena \cite{MadeoNeffGhibaW}. 

The plan of the paper is now the following.  In Section 2 we introduce the initial-boundary value problem arising in the model when non-standard inertia terms are present and we prove the existence of the solution. The used approach is not  common  in elastodynamics, it is based on the Banach-fixed point theorem  and it was suggested by an approach from thermo-visco-plasticity \cite{chelminski2016renormalized}. We  mention that  the  usual approaches based on the semigroup theory of linear operators or on the Galerkin method  seem to not be able to lead to an existence result without assuming  a priori some compatibilities between the domains of the operators involved in our partial differential equations. We avoid these assumptions, which, from a mechanical point of view lead to a very particular model  which is not capable to describe the behavior of waves in meta-materials, i.e. it does not correspond to  the main raison  d'\^{e}tre of our model.  In the last section,  we introduce a simplified model and we prove that the corresponding  initial-boundary value problem is still well-posed.

\section{Relaxed micromorphic models including micro-inertia terms}\setcounter{equation}{0}

We introduce a new model which takes into account the influence of some nonstandard micro-inertia terms on the behavior of the solution of the obtained initial-boundary value problem.

We consider a micromorphic continuum which occupies a bounded domain $\Omega$  having a piecewise smooth
surface $\partial \Omega$.  The motion of the body is referred to a fixed system of rectangular Cartesian axes $Ox_i$, $(i=1,2,3)$. Throughout this paper {(if we do not specify otherwise)} Latin subscripts take the values $1,2,3$.

We denote by $\mathbb{R}^{3\times 3}$ the set of real $3\times 3$ matrices.   For all $X\in\mathbb{R}^{3\times3}$ we set $\sym\, X=\frac{1}{2}(X^T+X)$ and $\skew X=\frac{1}{2}(X-X^T)$.
For $a,b\in\mathbb{R}^3$ we let $\langle {a},{b}\rangle_{\mathbb{R}^3}$  denote the scalar product on $\mathbb{R}^3$ with
associated vector norm $\norm{a}_{\mathbb{R}^3}^2=\Mprod{a}{a}_{\mathbb{R}^3}$.
The standard Euclidean scalar product on $\mathbb{R}^{3\times 3}$ is given by
$\Mprod{X}{Y}_{\mathbb{R}^{3\times3}}=\tr({X Y^T})$, and thus the Frobenius tensor norm is
$\|{X}\|^2=\langle{X},{X}\rangle_{\mathbb{R}^{3\times3}}$. In the following we omit the index
$\mathbb{R}^3,\mathbb{R}^{3\times3}$. The identity tensor on $\mathbb{R}^{3\times3}$ will be denoted by $\id$, so that
$\tr({X})=\langle{X},{\id}\rangle$. 
We adopt the usual abbreviations of Lie-algebra theory, i.e.,
$\so(3):=\{X\in\mathbb{R}^{3\times3}\;|X^T=-X\}$ is the Lie-algebra of  skew symmetric tensors
and $\Sym(3)$  denotes the set of symmetric tensors.

By $\Co$ we denote  the set of smooth functions with compact support in $\Omega$.
All the usual Lebesgue spaces of square integrable functions, vector or tensor fields on $\Omega$ with values in $\mathbb{R}$, $\mathbb{R}^3$ or $\mathbb{R}^{3\times 3}$, respectively will be generically denoted by $L^2(\Omega)$. Moreover, we use the standard Sobolev spaces \cite{Adams75,Leis86,Raviart79}
\begin{align}
&{\rm H}^1(\Omega)=\{u\in L^2(\Omega)\, |\, {\rm grad}\, u\in L^2(\Omega)\},\qquad \quad\  \|u\|^2_{{\rm H}^1(\Omega)}:=\|u\|^2_{L^2(\Omega)}+\|{\rm grad}\, u\|^2_{L^2(\Omega)}\, ,\notag\\
&{\rm H}({\rm curl};\Omega)=\{v\in L^2(\Omega)\, |\, {\rm curl}\, v\in L^2(\Omega)\}, \qquad \|v\|^2_{{\rm H}({\rm curl};\Omega)}:=\|v\|^2_{L^2(\Omega)}+\|{\rm curl}\, v\|^2_{L^2(\Omega)}\, ,\notag
\end{align}
of functions $u$ or vector fields $v$, respectively. Furthermore, we introduce their  {closed} subspaces ${\rm H}_0^1(\Omega)$, and ${\rm H}_0({\rm curl};\Omega)$ as the closure with respect to the associated graph norms of $C_0^\infty(\Omega)$. Roughly speaking, ${\rm H}_0^1(\Omega)$ is the subspace of functions $u\in {\rm H}^1(\Omega)$ 
which are
zero on
$\partial \Omega$, while ${\rm H}_0({\rm curl};\Omega)$ is the subspace of vectors $v\in{\rm H}({\rm curl};\Omega)$ which are normal at $\partial \Omega$ (see \cite{NeffPaulyWitsch,NPW2,NPW3}). For vector fields $v$ with components in ${\rm H}^{1}(\Omega)$ and tensor fields $P$ with rows in ${\rm H}({\rm curl}\,; \Omega)$, i.e.,
{\begin{align*}
	v=\left(
	v_1,   v_2,    v_3
	\right)^T,\quad  v_i\in {\rm H}^{1}(\Omega),
	\ \quad
	P=\left(
	P_1^T,    P_2^T,    P_3^T
	\right)^T,\, \quad P_i\in {\rm H}({\rm curl}\,; \Omega)\,,
	\end{align*}
	we define
	\begin{align*}
	\nabla\,v:=\left(
	({\rm grad}\,  v_1)^T,
	({\rm grad}\, v_2)^T,
	({\rm grad}\, v_3)^T
	\right)^T,\qquad  {\rm Curl}\,P:=\left(
	({\rm curl}\, P_1)^T,
	({\rm curl}\,P_2)^T,
	({\rm curl}\,P_3)^T
	\right)^T.
	\end{align*}}
We note that $v$ is a vector field, whereas $P$, ${\rm Curl}\, P$ and $\nabla\, v$ are second order tensor fields. The corresponding Sobolev spaces will be denoted by
$
{\rm H}^1(\Omega) \ \ \text{and}\  \ {\rm H}({\rm Curl}\,; \Omega)\, ,
$ and $
{\rm H}^1_0(\Omega) \ \ \text{and}\  \ {\rm H}_0({\rm Curl}\,; \Omega)\, ,
$ respectively.

For a  fourth order tensor $\mathbb{C}$ and $X\in \mathbb{R}^{3\times 3}$, we have $\C.  X\in \mathbb{R}^{3\times 3}$ with the components
$
(\C. X)_{ij}=\mathbb{C}_{ijkl}\,X_{kl}\, , 
$ while for a sixth order tensor $\mathbb{L}$  we consider
$
\mathbb{L}. Z\in \mathbb{R}^{3\times 3\times 3}$ for all $Z\in \mathbb{R}^{3\times 3\times 3}$, $(\mathbb{L}. Z)_{ijk}=\mathbb{L}_{ijkmnp}Z_{mnp},
$
{where  Einstein's summation rule is used}.

\subsection{Description of the mechanical model}

 The mechanical
model  is formulated in the variational context. This means that we
consider an action functional on an appropriate function-space. 
The space of configurations
of the problem is 
\[
\mathcal{Q}:=\left\{ \left(\u,\P\right)\in C^{1}\left(\overline{\Omega}\times [0,T],\R^{3}\right)\times C^{1}\left(\overline{\Omega}\times  [0,T],\R^{3\times3}\right):\left(\u,\P\right)\textrm{ verifies conditions }\left(\mathsf{B}_{1}\right)\textrm{ and }\left(\mathsf{B}_{2}\right)\right\} 
\]
where
\begin{itemize}
	\item $\left(\mathsf{B}_{1}\right)$ are the boundary conditions $u\left(x,t\right)=\varphi\left(x,t\right)$
	and $P_{i}\left(x,t\right)\times n=\psi_{i}\left(x,t\right)$, $i=1,2,3$,
	$\quad\left(x,t\right)\in\partial\Omega\times\left[0,T\right]$, where
	$n$ is the unit outward normal vector on $\partial\Omega\times\left[0,T\right]$,
	$P_{i},\,i=1,2,3$ are the rows of $P$ and $\varphi,\psi_{i}$ are
	prescribed   {continuous} functions and $[0,T]$ is the time interval;
	\item $\left(\mathsf{B}_{2}\right)$ are the initial conditions $\left.\u\right|_{t=0}=\u_{0},\left.\u_{,t}\right|_{t=0}=\underline{\u}_{0},\left.P\right|_{t=0}=P_{0},\left.P_{,t}\right|_{t=0}=\underline{P}_{0}\;\textrm{in }\Omega$,
	where $u_{0}\left(x\right),\underline{\u}_{0}\left(x\right),$ $P_{0}\left(x\right),\underline{P}_{0}\left(x\right)$
	are prescribed smooth functions.
\end{itemize}
The action functional $\mathscr{A}:\mathcal{Q}\to\R,$ is the sum
of the internal and external action functionals $\mathscr{A}_{\mathscr{L}}^{int},\mathscr{A}^{ext}:\mathcal{Q}\to\R$
defined as follows
\begin{align}
\mathscr{A}_{\mathscr{L}}^{int}\left[\left(\u,\P\right)\right] & := \int_0^T\int_{\Omega}\L\left(\u_{,t},\P_{,t},\nabla\u_{,t},\curl\,P_{,t},\nabla\u,\P,\curl\,P\right)dv\,dt,\\
\mathscr{A}^{ext}\left[\left(\u,\P\right)\right] & := \int_0^T\int_{\Omega}\left(\left\langle f,u\right\rangle +\left\langle M,P\right\rangle \right)dv\,dt,\nonumber 
\end{align}
where $\L$ is the Lagrangian density of the system and $f,M$
are the body force and double body force, and comma followed by a subscript denotes the time derivative. We recommend also the paper by Germain \cite{germain1973method} for some explanations about the physical significance of the involved quantities.  In order to find the stationary points of
the action functional, we have to calculate its first variation:
\begin{gather*}
\delta\mathscr{A}=\delta\mathscr{A}_{\L}^{int}=\delta \int_0^T\int_{\Omega}\L\left(\u_{,t},\P_{,t},\nabla\u_{,t},\curl\,P_{,t},\nabla\u,\P,\curl\,P\right)dv\,dt.
\end{gather*}
For the Lagrangian energy density we assume the standard split in
kinetic minus potential energy density:

\[
\mathscr{L}\left(\u_{,t},\P_{,t},\nabla\u_{,t},\curl\,P_{,t},\nabla\u,\P,\curl\,\p\right)=K\left(\u_{,t},\P_{,t},\nabla\u_{,t},\curl\,P_{,t}\right)-W\left(\nabla\u,\P,\curl\,\p\right),
\]

In general anisotropic linear elastic relaxed micromorphic homogeneous media,  we consider 
that the kinetic energy and the potential energy density have the following
expressions 
\begin{align*}
K\left(\u_{,t},\P_{,t},\nabla\u_{,t},\curl\,P_{,t}\right) & =\frac{1}{2}\left\langle \rho\,u_{,t},u_{,t}\right\rangle +\frac{1}{2}\left\langle J\:\P_{,t},\P_{,t}\right\rangle+\frac{1}{2}\left\langle \widetilde{\mathbb{C}}_{e}.\,\sym\left(\grad\u_{,t}-P_{,t}\right),\sym\left(\grad\u_{,t}-P_{,t}\right)\right\rangle \\&\qquad+\frac{1}{2}\left\langle \widetilde{\mathbb{C}}_{\rm c}.\:\skew\left(\grad\u_{,t}-P_{,t}\right),\skew\left(\grad\u_{,t}-P_{,t}\right)\right\rangle \\
& \qquad+\frac{1}{2}\left\langle \widetilde{\mathbb{C}}_{\rm micro}.\,\sym\,P_{,t},\sym\,P_{,t}\right\rangle + {\mu}\,\frac{L_{c}^{2}}{2}\left\langle \widetilde{\mathbb{L}}_{\textrm{aniso}}.\,\curl\,P_{,t},\curl\,P_{,t}\right\rangle, \\
W\left(\nabla\u,\P,\curl\,\p\right) & =\underbrace{\frac{1}{2}\left\langle \mathbb{C}_{e}.\,\sym\left(\grad\u-P\right),\sym\left(\grad\u-P\right)\right\rangle }_{\textrm{anisotropic elastic - energy}}+\underbrace{\frac{1}{2}\left\langle {\mathbb{C}}_{\rm c}.\:\skew\left(\grad\u-P\right),\skew\left(\grad\u-P\right)\right\rangle }_{\textrm{invariant local anisotropic rotational elastic coupling}}\\
& \qquad+\underbrace{\frac{1}{2}\left\langle \cm.\,\sym\,P,\sym\,P\right\rangle}_{\textrm{micro - self - energy}}+\underbrace{\: {\mu}\,\frac{L_{c}^{2}}{2}\left\langle {\mathbb{L}}_{\textrm{aniso}}.\,\curl\,P,\curl\,P\right\rangle }_{\textrm{ curvature-energy}},
\end{align*}
where
\[
\begin{cases}
\C_{e}, \widetilde{\C}_{e},\cm,  \widetilde{\C}_{\rm micro}:\Sym(3)\fr\Sym(3) & \textrm{are dimensionless  \ensuremath{4^{th}}order tensors  }
\\& \textrm{with 21 independent components},\vspace{2mm}\\
{\C}_{\rm c}, \widetilde{\C}_{\rm c}:\so(3)\fr\so(3) & \textrm{are dimensionless \ensuremath{4^{th}} order tensors},\\& \textrm{with 6 independent components},\vspace{2mm}\\
{\mathbb{L}}_{\textrm{aniso}}, \widetilde{\mathbb{L}}_{\textrm{aniso}}:\R^{3\times3}\fr\R^{3\times3} & \textrm{are dimensionless \ensuremath{4^{th}} order tensors},\\
&\textrm{with 45 independent components},
\end{cases}
\]
the positive constants $\rho, J>0$  are the macro-inertia and micro-inertia density, $L_{c}>0$ is the characteristic length of the relaxed micromorphic
model and $\mu> 0$ is a parameter used for dimension compatibility of the involved terms. The limit case $L_{c}\to 0$ corresponds to considering very large specimens of a microstructured meta-material.

For simplicity,  we  omit  the mixed terms in the form of our energies. One reason is that as long as there exist no clear mechanical interpretation of the influence of these terms on the process we want to model, we decided to keep the formulation as simple as possible. If, from some practical problems, the presence of these terms is requested, then they may be included. Another motivation, and the most important related to  the present paper, is that our mathematical analysis can be extended in a straightforward manner to the case when the mixed terms are also present in the total energy.

Compared to the classical Mindlin-Eringen model,  the curvature dependence is reduced to a dependence only on the {\it micro-dislocation tensor} $\alpha:=-\Curl P\in\mathbb{R}^{3\times 3}$ instead of
$\gamma=\nabla P\in\mathbb{R}^{27}=\mathbb{R}^{3\times 3\times 3}$. Doing so, the first main advantage of the relaxed micromorphic model is that the number of constitutive coefficients is drastically reduced. A second strong point of the relaxed model is that it is possible to show that in the limit case $L_c \to 0$ (which corresponds to considering very large specimens of a microstructured meta-material) the meso- and micro-coefficients $\mathbb{C}_{\rm e}$ and $\mathbb{C}_{\rm micro}$ of the relaxed model can be put in direct relation with the macroscopic stiffness of the medium via a fundamental homogenization formula, in contrast to the Eringen-Mindlin theory \cite{Eringen99,Mindlin64}, where it is  impossible to obtain this kind of results, see \cite{barbagallo2017transparent,neff2017real}.

   {We point out another important aspect, an existence result based on the assumption that the internal energy  is positive definite covers a particular situation when the energy from the classical Eringen-Mindlin micromorphic theory is positive semi-definite.} Moreover, the case ${\mathbb{C}}_{\rm c}=0$, i.e. the situation in which the force stress tensor is symmetric (see \eqref{eqrelax}),  or the situation when ${\mathbb{C}}_{\rm c}$ is only positive semi-definite are also covered by our analysis, and the variational setting  allows to prescribe tangential boundary conditions, i.e. ${P}_i({x},t)\times \,n(x) =0, \  i=1,2,3, 
\ ({x},t)\in\partial \Omega\times [0,T]$.

  {Only for the sake of simplicity, in the rest of the paper we assume  that the constitutive coefficients ${\mathbb{C}_{\rm e}}, {{\mathbb{C}}_{\rm micro}}, {\mathbb{C}}_{\rm c} ,{\mathbb{L}}_c$, ${\widetilde{\mathbb{C}}_{\rm e}}, {\widetilde{\mathbb{C}}_{\rm micro}}, \widetilde{\mathbb{C}}_{\rm c} ,\widetilde{\mathbb{L}}_c$  are constant  and they
have the following symmetries
\begin{align}\label{simetries}
&(\mathbb{C}_{\rm e})_{ijrs}=(\mathbb{C}_{\rm e})_{rsij}=(\mathbb{C}_{\rm e})_{jirs},\quad\quad \quad\quad \quad\quad \ \  ({\mathbb{C}}_{\rm c})_{ijrs}= -({\mathbb{C}}_{\rm c})_{jirs}= ({\mathbb{C}}_{\rm c})_{rsij},\notag\\&({\mathbb{C}}_{\rm micro})_{ijrs}=({\mathbb{C}}_{\rm micro})_{rsij}=({\mathbb{C}}_{\rm micro})_{jirs},\quad\quad  ({\mathbb{L}_{\rm aniso}})_{ijrs}=({\mathbb{L}_{\rm aniso}})_{rsij}\,,\\
&(\widetilde{\mathbb{C}}_{\rm e})_{ijrs}=(\widetilde{\mathbb{C}}_{\rm e})_{rsij}=(\widetilde{\mathbb{C}}_{\rm e})_{jirs},\quad\quad \quad\quad \quad\quad \ \   (\widetilde{\mathbb{C}}_{\rm c})_{ijrs}= -(\widetilde{\mathbb{C}}_{\rm c})_{jirs}= (\widetilde{\mathbb{C}}_{\rm c})_{rsij},\notag\\&(\widetilde{\mathbb{C}}_{\rm micro})_{ijrs}=(\widetilde{\mathbb{C}}_{\rm micro})_{rsij}=(\widetilde{\mathbb{C}}_{\rm micro})_{jirs},\quad\quad  (\widetilde{\mathbb{L}}_{\rm aniso})_{ijrs}=(\widetilde{\mathbb{L}}_{\rm aniso})_{rsij}\,.\notag
\end{align} 
 {The case of inhomogeneous media may be treated, under the supplementary assumption that the constitutive coefficients are bounded, as functions on ${\Omega}$.}}

We find, after considering the first variation of the action functional $\mathscr{A}$, that the general equations of the relaxed micromorphic model including the nonstandard inertia terms are 
\begin{align}\label{eqrelax}
 &\text{balance of forces}:\notag\\
  \rho\,u_{,tt}&-\dvg[ \widetilde{\C}_e. \sym(\nabla u_{,tt}-P_{,tt})+\widetilde{\C}_{\rm c}.\,{\rm skew}(\nabla u_{,tt}-P_{,tt})]=\notag\\&\ \ \ \ \dvg[ \underbrace{\mathbb{C}_{\rm e}. \sym(\nabla u-P)+{\C}_{\rm c}.\,{\rm skew}(\nabla u-P)}_{\sigma-\text{force stress tensor}}]+f\, , \quad\qquad \qquad\qquad \qquad \qquad\ \ \quad \quad  \notag\\\\
 & \text{balance of moment stresses}:\notag
  \\
J\,P_{,tt}&+ {\mu}\,L_{c}^{2}\,\crl[ \widetilde{\mathbb{L}}_{\rm aniso}.\crl\, P_{,tt}]-\widetilde{\C}_e. \sym (\nabla u_{,tt}-P_{,tt})-\widetilde{\C}_{\rm c}.\,{\rm skew}(\nabla u_{,tt}-P_{,tt})+\widetilde{\C}_{\rm micro}. \sym P_{,tt}=\notag\\&\quad -{\mu}\,L_{c}^{2}\, \crl[ {\mathbb{L}}_{\rm aniso}.\crl\, P]+\mathbb{C}_{\rm e}. \sym (\nabla u-P)+{\C}_{\rm c}.\,{\rm skew}(\nabla u-P)-{\C}_{\rm micro}. \sym P+M. \notag
\end{align}

To the above system of partial differential equations, we adjoin  the  boundary conditions
\begin{align} \label{bc}
{u}({x},t)=0, \ \ \  \text{and the {\it tangential condition}}  \quad {P}_i({x},t)\times n(x) =0, \ \ \ i=1,2,3, \ \ \
\ ({x},t)\in\partial \Omega\times [0,T],
\end{align}%
where  $P_i, i = 1, 2, 3$ are
the rows of $P$ and the  initial conditions
\begin{align}\label{ic}
{u}({x},0)={u}_0(
x),\quad \quad \quad{u}_{,t}({x},0)=\underline{u}_0(
x),\quad \quad \quad
{P}({x},0)={P}_0(
x), \quad \quad \quad {P}_{,t}({x},0)=\underline{P}_0(
x),\ \ \text{\ \ }{x}\in {\Omega},
\end{align}%
where  ${u}_0, \underline{u}_0, {P}_0$ and $\underline{P}_0$ are prescribed functions.

  {Regarding  the  tensors $\widetilde{\mathbb{C}}_{\rm e}$, and $\widetilde{\mathbb{L}}_{\rm aniso}$, we assume that they are positive definite, i.e. 
	there exist  positive numbers ${\widetilde{\mu}_{\rm e}^m}, {\widetilde{\mu}_{\rm e}^M}>0$ (the maximum and minimum elastic moduli for $\widetilde{\C}_e$),   $\widetilde{L}_{\rm c}^M$, $\widetilde{L}_{\rm c}^m>0$ (the maximum and minimum moduli for $\widetilde{\mathbb{L}}_{\rm aniso}$) 
	\begin{align}\label{posdef}
	{\widetilde{\mu}_{\rm e}^m}\|X \|^2\leq \langle\,\widetilde{\mathbb{C}}_{\rm e}. X,&X\rangle\leq {\widetilde{\mu}_{\rm e}^M}\|X \|^2\,\quad\quad  \notag\ \ \text{for all }\ \ X\in\Sym(3),
	\\
	\widetilde{L}_c^m\|X \|^2\leq \langle \widetilde{\mathbb{L}}_{\rm aniso}. X,&X\rangle\leq \widetilde{L}_c^M\| X\|^2\,\quad \quad\  \ \text{for all }\ \ X\in\mathbb{R}^{3\times 3}.
	\end{align}
	We also suppose that the constitutive tensors  $\widetilde{\C}_{\rm c}$ and  $\widetilde{\C}_{\rm micro}$ are positive semi-definite, i.e 
	\begin{align}\label{posdef3}
	0&\leq \langle\,\widetilde{\C}_{\rm c}. X,X\rangle\qquad\ \, \quad  \ \ \text{for all }\ \ \, X\in\so(3),\\
	0&\leq\langle \widetilde{\C}_{\rm micro}. X,X\rangle  \quad \ \ \ \  \ \text{for all }\ \ X\in\Sym(3),\notag
	\end{align}
	which means that  we do not exclude the possibility that these constitutive tensors  vanish.
}

 {Since the media under consideration is  homogeneous,  the quadratic forms defined by the fourth order elasticity tensors ${\C}_e$, ${\C}_{\rm c}$, ${\C}_{\rm micro}$ and ${\mathbb{L}}_{\rm aniso}$ and by the tensors $\widetilde{\C}_{\rm c}$ and  $\widetilde{\C}_{\rm micro}$, respectively, are bounded,} i.e. 
there exist   ${\mu_{\rm e}^M}\geq 0$, ${\mu_{\rm c}^M}\geq 0$,  $	\mu_{\rm micro}^M\geq 0$, $L_{\rm c}^M\geq 0$, ${\widetilde{\mu}_{\rm c}^M}\geq 0$ and ${\widetilde{\mu}_{\rm micro}^M}\geq 0$ such that
\begin{align}\label{posdef0}
\langle\,\mathbb{C}_{\rm e}. X,&X\rangle\leq {\mu_{\rm e}^M}\|X \|^2\,\quad\quad  \notag\ \ \text{for all }\ \ X\in\Sym(3),
\\
\langle\,{\C}_{\rm c}. X,&X\rangle\leq {\mu_{\rm c}^M}\|X \|^2\,\quad \ \ \ \ \  \text{for all }\ \ X\in\so(3),\notag\\
\langle {\mathbb{L}}_{\rm aniso}. X,&X\rangle\leq L_c^M\| X\|^2\,\quad \quad\  \ \text{for all }\ \ X\in\mathbb{R}^{3\times 3},\\
\langle {\C}_{\rm micro}. X,&X\rangle\leq 	\mu_{\rm micro}^M\| X\|^2\, \quad \ \  \text{for all }\ \ X\in\Sym(3),\notag\\
\langle\,\widetilde{\C}_{\rm c}. X,&X\rangle\leq {\widetilde{\mu}_{\rm c}^M}\,\|X\|^2\qquad\ \,  \ \text{for all }\ \ \, X\in\so(3),\notag\\
\langle \widetilde{\C}_{\rm micro}. X,&X\rangle\leq {\widetilde{\mu}_{\rm micro}^M}\,\|X\|^2 \quad \ \ \text{for all }\ \ X\in\Sym(3).\notag
\end{align}
In the rest of this paper we denote tensors having this property as bounded tensors. In our assumptions  the elastic constitutive tensors  may  be negative definite or  positive semi-definite, or they may even vanish.

 {In the next subsection, we will prove that the above model is well-posed. }

 \subsection{Existence of the solution}

 We introduce two bilinear forms
 $\mathcal{W}_1,\mathcal{W}_2:({H}_0^1(\Omega)
 \times{H}_0(\Curl;\Omega))\times ({H}_0^1(\Omega)
 \times{H}_0(\Curl;\Omega))\rightarrow\mathbb{R}$
 \begin{align}
 \mathcal{W}_1((u,P),(\varphi,\Phi))
 &=\dd\int _\Omega\biggl(\rho\,\langle u,\varphi\rangle+J\,\langle P,\Phi\rangle+\langle \widetilde{\mathbb{C}}_{\rm e}. \sym(\nabla u-P), \sym(\nabla \varphi-\Phi)\rangle,\notag\\&\qquad\quad  +\langle \widetilde{\mathbb{C}}_{\rm c}. \skew(\nabla u-P), \skew(\nabla \varphi-\Phi)\rangle +\langle \widetilde{\mathbb{C}}_{\rm micro}. \sym P, \sym \Phi\rangle\notag \\&\qquad \quad +{\mu}\,L_{c}^{2}\,\langle \widetilde{\mathbb{L}}_{\rm aniso}. \Curl\, P, \Curl\, \Phi\rangle\biggl)dv,\\
 \mathcal{W}_2((u,P),(\varphi,\Phi))
 &=\dd\int _\Omega\biggl(\langle \mathbb{C}_{\rm e}. \sym(\nabla u-P), \sym(\nabla \varphi-\Phi)\rangle+\langle \mathbb{C}_{\rm c}. \skew(\nabla u-P), \skew(\nabla \varphi-\Phi)\rangle\,\notag\\&\qquad \quad +\langle \mathbb{C}_{\rm micro}. \sym P, \sym \Phi\rangle +{\mu}\,L_{c}^{2}\,\langle \mathbb{L}_{\rm aniso}. \Curl\, P, \Curl\, \Phi\rangle\biggl)dv\, \notag
 \end{align}
 and for each  $(f,M)\in H^{-1}(\Omega)\times (H_0(\curl;\Omega))^*$ we consider the linear operator  $l^{(f,M)}\in H^{-1}(\Omega)\times (H_0(\curl;\Omega))^*$ defined by
\begin{align*}
l^{(f,M)}:{H}_0^1(\Omega)
\times{H}_0(\Curl;\Omega)\rightarrow\mathbb{R}, \quad  l^{(f,M)}(\varphi,\Phi)=\langle f, \varphi\rangle_{H^{-1}(\Omega), H_0^1(\Omega)}+\langle M, \Phi\rangle_{(H_0(\curl;\Omega))^*,H_0(\curl;\Omega)},
 \end{align*} 
 where $H^{-1}(\Omega)$ and $(H_0(\curl;\Omega))^*$ are the dual spaces of  $H^1_0(\Omega)$ and $H_0(\curl;\Omega)$, respectively. We equip the product space ${H}_0^1(\Omega)
 \times{H}_0(\Curl;\Omega)$ with the norm
 $$
 \|(u,P)\|_{{H}_0^1(\Omega)
 	\times{H}_0(\Curl;\Omega)}=\left(\|u\|^2_{{H}_0^1(\Omega)
 	}+\|P\|^2_{{H}_0(\Curl;\Omega)}\right)^{1/2}.
 $$

 For every $(f,M)\in C([0,T];H^{-1}(\Omega)\times (H_0(\curl;\Omega))^*$, the pair $(u,P)\in C^2([0,T];H_0^1(\Omega)\times H_0(\curl;\Omega))$ is a weak solution of the problem \eqref{eqrelax}--\eqref{ic}  provided 
   {\begin{align}\label{wformulation}
 \mathcal{W}_1((u_{,tt}(t),P_{,tt}(t)),(\varphi
,\Phi))+\mathcal{W}_2((u(t),P(t)),(\varphi
,\Phi))= l^{(f(t),M(t))}(\varphi,\Phi)
 \end{align}}
 for each $(\varphi, \Phi)\in H_0^1(\Omega)\times H_0(\curl;\Omega)$  and for all $t\in [0,T]$,
 and it satisfies\begin{align}\label{icw}
 &{u}({x},0)={u}_0(
 x),\quad \quad \quad{u}_{,t}({x},0)=\underline{u}_0(
 x),\quad \quad \quad
 {P}({x},0)={P}_0(
 x), \quad \quad \quad{P}_{,t}({x},0)=\underline{P}_0(
 x),\ \ \text{\ \ }{x}\in {\Omega}.
 \end{align}%
 
 The proof that any classical solution is a weak solution follows using analogue calculations as in \cite{NeffGhibaMadeoLazar}.
 
 In order to prove the existence and uniqueness of a weak solution of  problem \eqref{eqrelax}, we follow the strategy  proposed in \cite{chelminski2016renormalized} and we first prove the following lemma.
  \begin{lemma}
   {	Assume that
 	\begin{itemize}
  \item[i)] the constitutive tensors satisfy the symmetry relations \eqref{simetries}, respectively,
  \item[ii)]  	$\mathbb{C}_{\rm e}, \mathbb{C}_{\rm micro},\mathbb{C}_{\rm c}$ and $\mathbb{L}_{\rm aniso}$ are bounded tensors,
  \item[iii)]  $\widetilde{\mathbb{C}}_{\rm e}$ and $ \widetilde{\mathbb{L}}_{\rm aniso}$ are positive definite tensors,
  \item[iv)]
 	$\widetilde{\mathbb{C}}_{\rm micro}$ and $\widetilde{\mathbb{C}}_{\rm c}$ are positive semi-definite tensors,
 		\item[v)] 
 	$(f,M)\in C([0,T];H^{-1}(\Omega)\times(H_0({\rm Curl};\Omega))^*)$,   
 	\item[vi)]  $(u_0,P_0),$ $ (\underline{u}_0,\underline{P}_0)\in H_0^1(\Omega)\times H_0({\rm Curl};\Omega)$,
 	\item[vii)]    $\rho, J, L_c,\mu>0$,
 \end{itemize}
then,  for all $(v,Q)\in C([0,T];H_0^1(\Omega)\times H_0({\rm Curl};\Omega))$, there exists a unique   function $(u,P)\in C^2([0,T];H_0^1(\Omega)\times H_0({\rm Curl};\Omega))$ such that
 	\begin{align}\label{exsis10}
 	\mathcal{W}_1((u_{,tt}(t),P_{,tt}(t)),(\varphi
 	,\Phi))=-\mathcal{W}_2((v(t),Q(t)),(\varphi
 	,\Phi))+ l^{(f(t),M(t))}(\varphi,\Phi),
 	\end{align}
 	for all $(\varphi, \Phi)\in H_0^1(\Omega)\times H_0({\rm Curl};\Omega)$  and for all $t\in [0,T]$,
 	and it satisfies\begin{align}\label{icw1}
 	&{u}({x},0)={u}_0(
 	x),\quad \quad \quad{u}_{,t}({x},0)=\underline{u}_0(
 	x),\quad \quad \quad
 	{P}({x},0)={P}_0(
 	x), \quad \quad \quad{P}_{,t}({x},0)=\underline{P}_0(
 	x),\ \ \text{\ \ }{x}\in {\Omega}.
 	\end{align}}
 \end{lemma}
 
 \begin{proof}
 
 Let us consider a fixed time $t\in [0,T]$.	 From the Cauchy-Schwarz inequality and Poincar\'e inequality, and since the constitutive coefficients satisfy \eqref{posdef0}, we find that for fixed $(w,R)\in H_0^1(\Omega)\times H_0({\rm Curl};\Omega)$ the map $ \mathcal{W}_2((w,R),\cdot):H_0^1(\Omega)\times H_0(\curl;\Omega)\to \mathbb{R}$ is bounded, since 
 	\begin{align}
 	&\|\sym(\nabla   \varphi-\Phi)\|^2\leq 2\left( \|\sym \nabla   \varphi\|^2+\|\sym \Phi\|^2\right),\qquad 
 	\|\sym \nabla   \varphi\|^2\leq \|\nabla   \varphi\|^2,\quad\quad
 	\|\sym \Phi\|^2\leq \|\Phi\|^2,\\
 	&  {\|\skew(\nabla   \varphi-\Phi)\|^2\leq 2\left( \|\skew \nabla   \varphi\|^2+\|\skew \Phi\|^2\right)},\ \ \ \ 
 	\|\skew \nabla   \varphi\|^2\leq \|\nabla   \varphi\|^2,\ \ \ \ \ \,
 	\|\skew \Phi\|^2\leq \|\Phi\|^2,\notag
 	\end{align}
 	for all $  \varphi\in H_0^1(\Omega)$ and for all $\Phi\in L^2(\Omega)$. 
 	Similarly  it follows that $\mathcal{W}_1$ is bounded.
 	On the other hand
 	\begin{align}
 	\mathcal{W}_1((  \varphi,\Phi),(  \varphi,\Phi))
 	&=\dd\int _\Omega\biggl(\rho\,\|   \varphi\|^2+J\,\| \Phi\|^2+\langle \widetilde{\mathbb{C}}_{\rm e}. \sym(\nabla   \varphi-\Phi), \sym(\nabla {  \varphi}-{\Phi})\rangle\,\notag\\&\qquad\quad   +\langle \widetilde{\mathbb{C}}_{\rm c}. \skew(\nabla   \varphi-\Phi), \skew(\nabla {  \varphi}-{\Phi})\rangle+\langle \widetilde{\mathbb{C}}_{\rm micro}. \sym \Phi, \sym {\Phi}\rangle\\&\qquad\quad  +{\mu}\,L_{c}^{2}\,\langle \widetilde{\mathbb{L}}_{\rm aniso}. \Curl\, \Phi, \Curl\, {\Phi}\rangle\!\biggl) dv\,\notag\\
 	&  { \geq \!\dd\int _\Omega\!\biggl(\!\rho\,\|   \varphi\|^2+J\,\| \Phi\|^2+\langle \widetilde{\mathbb{C}}_{\rm e}. \sym(\nabla   \varphi-\Phi), \sym(\nabla {  \varphi}-{\Phi})\rangle+{\mu}\,L_{c}^{2}\,\langle \widetilde{\mathbb{L}}_{\rm aniso}. \Curl\, \Phi, \Curl\, {\Phi}\rangle\biggl)\, dv}\notag
 	\\
 	&  {\geq\! \dd\int _\Omega\!\biggl(\!\rho\,\|   \varphi\|^2+\frac{J}{2}\,\| \Phi\|^2+\frac{J}{2}\,\| \sym\,\Phi\|^2+\widetilde{\mu}_{\rm e}^m\|\sym(\nabla   \varphi-\Phi)\|^2+{\mu}\,L_{c}^{2}\,\widetilde{L}_{\rm c}^m \|\Curl\, \Phi\|^2\!\biggl) dv}\notag
 	\end{align}
 	for all
 	$({  \varphi},{\Phi})\in H_0^1(\Omega)\times H_0({\rm Curl};\Omega)$, since, if it is present, the tensor   {$\widetilde{\mathbb{C}}_{\rm c}$} and   {$\widetilde{\mathbb{C}}_{\rm micro}$} are positive semi-definite,  while $\widetilde{\mathbb{L}}_{\rm aniso}$ and $\widetilde{\mathbb{C}}_{\rm c}$ are positive definite.

  {Since }
 \begin{align*}
 \| \sym \nabla   \varphi\|^2\leq \| \sym (\nabla   \varphi-\Phi)\|^2+\|\sym \Phi\|^2
 \end{align*}
 for all $  \varphi\in {\rm H}^1(\Omega)$ and $\Phi\in L^2(\Omega)$, we have that there exists  $c>0$ such that
 	\begin{align*}
 \mathcal{W}_1((  \varphi,\Phi),(  \varphi,\Phi))
 &\geq c\,\dd\int _\Omega\biggl(\|   \varphi\|^2+\| \Phi\|^2+  {\| \sym \nabla   \varphi\|^2}+\|\Curl\, \Phi\|^2\biggl)\, dv\\\
 &\geq c\,\dd\int _\Omega\biggl(\| \sym \nabla   \varphi\|^2+\| \Phi\|^2+\|\Curl\, \Phi\|^2\biggl)\, dv
 \end{align*}
 for all
 $(  \varphi,\Phi)\in H_0^1(\Omega)\times H_0(\curl;\Omega)$.
  Hence, the coercivity of 	 $\mathcal{W}_1$ is assured by the coercivity of the quadratic form from the right-hand side of the above inequality, which is a direct consequence of  the classical Korn inequality \cite{Neff00b}. 
  
 	Since the linear operator  
 	$l^{(f(t),M(t))}$  is bounded for a fixed time $t$ considered in the beginning of the proof,  using the Lax-Milgram theorem we obtain the existence  of a unique solution $(u^*(t),P^*(t))\in H_0^1(\Omega)\times H_0(\curl;\Omega)$ of the equation
 	\begin{align}\label{exsis11}
 	\mathcal{W}_1((u^*(t),P^*(t)),(\varphi
 	,\Phi))=  {-}\mathcal{W}_2((v(t),Q(t)),(\varphi
 	,\Phi))+ l^{(f(t),M(t))}(\varphi,\Phi)
 	\end{align}
 	for all $(\varphi, \Phi)\in H_0^1(\Omega)\times H_0(\curl;\Omega)$.
 From standard arguments it follows that if $(f,M)\in C([0,T];H^{-1}(\Omega)\times (H_0(\curl;\Omega))^*)$ and  $(v,Q)\in C([0,T];H_0^1(\Omega)\times H_0(\curl;\Omega)) $, then the solution $(u^*,P^*)$ of \eqref{exsis11} belongs to $C([0,T];H_0^1(\Omega)\times H_0(\curl;\Omega))$. Indeed, the difference of the solutions of \eqref{exsis11} which correspond to two times $t_1, t_2\in [0,T]$ satisfies
 \begin{align} \mathcal{W}_1((u^*(t_1)-u^*(t_2),&P^*(t_1)-P^*(t_2)),(\varphi
 ,\Phi))=  {-}\mathcal{W}_2((v(t_1)-v(t_2),Q(t_1)-Q(t_2)),(\varphi
 ,\Phi))\notag\\&+ \langle f(t_1)-f(t_2), \varphi\rangle_{H^{-1}(\Omega), H_0^1(\Omega)}+\langle M(t_1)-M(t_2), \Phi\rangle_{H_0^*(\curl;\Omega),H_0(\curl;\Omega)} 
 \end{align}
 for all $(\varphi, \Phi)\in H_0^1(\Omega)\times H_0(\curl;\Omega)$.
Using  the coercivity of $\mathcal{W}_1$ and the boundedness of $\mathcal{W}_2$  we obtain that 
\begin{align}
\|u^*(t_1)-u^*(t_2)\|_{H_0^1(\Omega)}^2&+\|P^*(t_1)-P^*(t_2)\|_{H_0({\rm Curl};\Omega)}^2\notag\\&\leq c\,(\|u^*(t_1)-u^*(t_2)\|_{H_0^1(\Omega)}^2+\|P^*(t_1)-P^*(t_2)\|_{H_0({\rm Curl};\Omega)}^2)^{1/2}\\
&\qquad [ (\|v(t_1)-v(t_2)\|_{H_0^1(\Omega)}^2+\|Q(t_1)-Q(t_2)\|_{H_0({\rm Curl};\Omega)}^2)^{1/2}\notag\\&\qquad \quad+(\|f(t_1)-f(t_2)\|_{H^{-1}(\Omega)}^2+\|M(t_1)-M(t_2)\|_{H^*(\curl;\Omega)}^2 )^{1/2}],\notag
\end{align}
where $c$ is a positive constant. Hence,  there is a positive constant $c>0$ such that
  {\begin{align*}
&\|u^*(t_1)-u^*(t_2)\|_{H_0^1(\Omega)}^2+\|P^*(t_1)-P^*(t_2)\|_{H_0({\rm Curl};\Omega)}^2\\&\leq c\left[ \|v(t_1)-v(t_2)\|_{H_0^1(\Omega)}^2+\|Q(t_1)-Q(t_2)\|_{H_0({\rm Curl};\Omega)}^2+\|f(t_1)-f(t_2)\|_{H^{-1}(\Omega)}^2+\|M(t_1)-M(t_2)\|_{H^*(\curl;\Omega)}^2\right] ,\notag
\end{align*}}
which proves the continuity.

 Now,  the unique solution $(u,P)\in C^2([0,T];H_0^1(\Omega)\times H_0(\curl;\Omega))$ of the problem defined by \eqref{exsis10} and \eqref{icw1} will be
 \begin{align}\label{ustar}
 u(t)=u_0+\int_0^t\left(\underline{u}_0+\int_0^s u^*(\xi) d\,\xi\right) ds,\qquad  P(t)=P_0+\int_0^t\left(\underline{P}_0+\int_0^s P^*(\xi) \,d\xi\right) ds
 \end{align}
 for all $t\in[0,T]$ and the proof is complete.
\end{proof}

 \begin{theorem}\label{theoremexist}
  { Assume that
 	\begin{itemize}
 		\item[i)] the constitutive tensors satisfy the symmetry relations \eqref{simetries}, respectively,
 		\item[ii)]  	$\mathbb{C}_{\rm e}, \mathbb{C}_{\rm micro},\mathbb{C}_{\rm c}$ and $\mathbb{L}_{\rm aniso}$ are bounded tensors,
 		\item[iii)]  $\widetilde{\mathbb{C}}_{\rm e}$ and $ \widetilde{\mathbb{L}}_{\rm aniso}$ are positive definite tensors,
 		\item[iv)]
 		$\widetilde{\mathbb{C}}_{\rm micro}$ and $\widetilde{\mathbb{C}}_{\rm c}$ are positive semi-definite tensors,
 		\item[v)] 
 		$(f,M)\in C([0,T];H^{-1}(\Omega)\times(H_0({\rm Curl};\Omega))^*)$,   
 		\item[vi)]  $(u_0,P_0),$ $ (\underline{u}_0,\underline{P}_0)\in H_0^1(\Omega)\times H_0({\rm Curl};\Omega)$,
 		\item[vii)]    $\rho, J, L_c,\mu>0$,
 	\end{itemize}
  then there exists a unique   solution $(u,P)\in C^2([0,T];H_0^1(\Omega)\times H_0({\rm Curl};\Omega))$ of the problem defined by \eqref{wformulation} and \eqref{icw}.}
 \end{theorem}
\begin{proof}
To prove this theorem, we will use the Banach fixed-point theorem for the  mapping $$\mathcal{L}:C([0,\delta];H_0^1(\Omega)\times H_0(\curl;\Omega))
\to C^2([0,\delta];H_0^1(\Omega)\times H_0(\curl;\Omega))\subset  C([0,\delta];H_0^1(\Omega)\times H_0(\curl;\Omega))$$ which, for fixed $(f,M)\in C([0,\delta];H^{-1}(\Omega)\times (H_0({\rm Curl};\Omega))^*)$, $(u_0,P_0),$ $ (\underline{u}_0,\underline{P}_0)\in H_0^1(\Omega)\times H_0(\curl;\Omega)$, maps
each $(v,Q)\in C([0,\delta];H_0^1(\Omega)\times H_0(\curl;\Omega))$ to the solution of the corresponding problem defined by \eqref{exsis10} and \eqref{icw1}, where $\delta>0$ will be suitably chosen. 

Let us consider $(v^{(1)},Q^{(1)}),(v^{(2)},Q^{(2)})\in C([0,\delta];H_0^1(\Omega)\times H_0({\rm Curl};\Omega))$ and their corresponding solutions $(u^{(1)},P^{(1)})$, $(u^{(2)},P^{(2)})\in C^2([0,\delta];H_0^1(\Omega)\times H_0({\rm Curl};\Omega))$ of the related problems defined by \eqref{exsis10} and \eqref{icw1}, i.e.
\begin{align*}
u^{(\alpha)}(t)=u_0+\int_0^t\left(\underline{u}_0+\int_0^s u^{(\alpha)*}(\xi) d\,\xi\right) ds,\quad  P^{(\alpha)}(t)=P_0+\int_0^t\left(\underline{P}_0+\int_0^s P^{(\alpha)*}(\xi) \,d\xi\right) ds,\quad \alpha=1,2,
\end{align*}
for all   {$t\in[0,\delta]$}, where
 $(u^{(\alpha)*},P^{(\alpha)*})\in C([0,\delta];H_0^1(\Omega)\times H_0({\rm Curl};\Omega))$ is the unique solution of the equation:
\begin{align*}
\mathcal{W}_1((u^{(\alpha)*}(t),P^{(\alpha)*}(t)),(\varphi
,\Phi))=  {-}\mathcal{W}_2((v(t),Q(t)),(\varphi
,\Phi))+ l^{(f(t),M(t))}(\varphi,\Phi)
\end{align*}
for all $(\varphi, \Phi)\in H_0^1(\Omega)\times H_0(\curl;\Omega)$ and for all   {$t\in[0,\delta]$}.

 Then, because the solutions correspond to the same initial conditions and forces, we have
\begin{align}
\max_{t\in [0,\delta]}(\|u^{(1)}(t)&-u^{(2)}(t)\|_{H_0^1(\Omega)}+\|P^{(1)}(t)-P^{(2)}(t)\|_{H_0({\rm Curl};\Omega)})\notag\\
&\leq \max_{t\in [0,\delta]}\int_0^t\int_0^s(\|u^{(1)*}(\xi)-u^{(2)*}(\xi)\|_{H_0^1(\Omega)}+\|P^{(1)*}(
\xi)-P^{(2)*}(\xi)\|_{H_0({\rm Curl};\Omega)})d\xi\, ds.
\end{align}
Since  the problem is linear, $(u^{(1)*}-u^{(2)*},P^{(1)*}-P^{(2)*})\in C([0,\delta];H_0^1(\Omega)\times H_0({\rm Curl};\Omega))	$ satisfies 
\begin{align}
\mathcal{W}_1((u^{(1)*}(t)-u^{(2)*}(t),P^{(1)*}(t)-P^{(2)*}(t)),(\varphi
,\Phi))=  {-}\mathcal{W}_2((v^{(1)}(t)-v^{(2)}(t),Q^{(1)}(t)-Q^{(2)}(t)),(\varphi
,\Phi))
\end{align}
for all $(\varphi, \Phi)\in H_0^1(\Omega)\times H_0(\curl;\Omega)$  and for all $t\in [0,\delta]$. The coercivity of $\mathcal{W}_1$ and the boundedness of $\mathcal{W}_2$ lead us to
\begin{align}
\|u^{(1)*}(t)&-u^{(2)*}(t)\|_{H_0^1(\Omega)}^2+\|P^{(1)*}(t)-P^{(2)*}(t)\|_{H_0({\rm Curl};\Omega)}^2\notag\\&\leq c\, (\|v^{(1)}(t)-v^{(2)}(t)\|_{H_0^1(\Omega)}^2+\|Q^{(1)}(t)-Q^{(2)}(t)\|_{H_0({\rm Curl};\Omega)}^2),
\end{align}
where $c$ is a positive constant. 
But
\begin{align}
\frac{1}{2}(\|u^{(1)*}(t)-u^{(2)*}(t)\|_{H_0^1(\Omega)}&+\|P^{(1)*}(t)-P^{(2)*}(t)\|_{H_0({\rm Curl};\Omega)})^2\notag\\
&\leq\|u^{(1)*}(t)-u^{(2)*}(t)\|_{H_0^1(\Omega)}^2+\|P^{(1)*}(t)-P^{(2)*}(t)\|_{H_0({\rm Curl};\Omega)}^2,
\end{align}
and therefore we deduce that 
\begin{align}
\|u^{(1)*}(t)-u^{(2)*}(t)\|_{H_0^1(\Omega)}&+\|P^{(1)*}(t)-P^{(2)*}(t)\|_{H_0({\rm Curl};\Omega)}\\&\leq \sqrt{2\, c}\, (\|v^{(1)}(t)-v^{(2)}(t)\|_{H_0^1(\Omega)}^2+\|Q^{(1)}(t)-Q^{(2)}(t)\|_{H_0({\rm Curl};\Omega)}^2)^{1/2}.\notag
\end{align}
Hence, we obtain 
\begin{align}
\max_{t\in [0,\delta]}(\|u^{(1)}(t)&-u^{(2)}(t)\|^2_{H_0^1(\Omega)}+\|P^{(1)}(t)-P^{(2)}(t)\|^2_{H_0({\rm Curl};\Omega)})^{1/2}\notag\\
&\leq\max_{t\in [0,\delta]}(\|u^{(1)}(t)-u^{(2)}(t)\|_{H_0^1(\Omega)}+\|P^{(1)}(t)-P^{(2)}(t)\|_{H_0({\rm Curl};\Omega)})\\
&\leq \sqrt{2\, c}
\,\max_{t\in [0,\delta]}\int_0^t\int_0^s (\|v^{(1)}(\xi)-v^{(2)}(\xi)\|_{H_0^1(\Omega)}^2+\|Q^{(1)}(\xi)-Q^{(2)}(\xi)\|_{H_0({\rm Curl};\Omega)}^2)^{1/2}d\xi\, ds
\notag\\
&\leq \delta^2\,\sqrt{2\, c}\, \max_{t\in [0,\delta]} (\|v^{(1)}(t)-v^{(2)}(t)\|_{H_0^1(\Omega)}^2+\|Q^{(1)}(t)-Q^{(2)}(t)\|_{H_0({\rm Curl};\Omega)}^2)^{1/2}.\notag 
\end{align}
Therefore
\begin{align}\label{constc}
\max_{t\in [0,\delta]}(\|u^{(1)}(t)&-u^{(2)}(t)\|^2_{H_0^1(\Omega)}+\|P^{(1)}(t)-P^{(2)}(t)\|^2_{H_0({\rm Curl};\Omega)})^{1/2}
\notag\\&\leq \delta^2\,c\, \max_{t\in [0,\delta]} (\|v^{(1)}(t)-v^{(2)}(t)\|_{H_0^1(\Omega)}^2+\|Q^{(1)}(t)-Q^{(2)}(t)\|_{H_0({\rm Curl};\Omega)}^2)^{1/2}, 
\end{align}
where $c$ is positive constant which is independent of time and also of the initial condition.

Hence, $\mathcal{L}$ is a contraction  on $C([0,\delta];H_0^1(\Omega)\times H_0({\rm Curl};\Omega))$ for   {$\delta=\frac{1}{2\sqrt{c}}$} with $c$ taken from the above estimate.  Hence, for   {$\delta=\frac{1}{2\sqrt{c}}$} ,
there exists a unique   function $(u,P)\in C([0,\delta];H_0^1(\Omega)\times H_0(\curl;\Omega))$ such that
\begin{align}
(u,P)=\mathcal{L}. (u,P)\in C^2([0,\delta];H_0^1(\Omega)\times H_0(\curl;\Omega)).
\end{align} Therefore, there exists a unique $(u,P)\in C^2([0,\delta];H_0^1(\Omega)\times H_0(\curl;\Omega))$ solution of the problem
\begin{align*}
\mathcal{W}_1((u_{,tt}(t),P_{,tt}(t),(\varphi
,\Phi))=  {-}\mathcal{W}_2((u(t),P(t),(\varphi
,\Phi))+ l^{(f(t),M(t)}(\varphi,\Phi)
\end{align*}
for all $(\varphi, \Phi)\in H_0^1(\Omega)\times H_0(\curl;\Omega)$  and for all $t\in [0,\delta]$,
which satisfies
\begin{align*}
&{u}({x},0)={u}_0(
x),\quad \quad \quad{u}_{,t}({x},0)=\underline{u}_0(
x),\quad \quad \quad
{P}({x},0)={P}_0(
x), \quad \quad \quad{P}_{,t}({x},0)=\underline{P}_0(
x),\ \ \text{\ \ }{x}\in {\Omega}
\end{align*}%
for all $t\in[0,\delta]$.

By repeating the above analysis, since the positive constant $c$ in \eqref{constc} does not depend on the initial data,  we  argue that 
there exists a unique solution  
$(\widetilde{u},\widetilde{P})\in C^2([\delta, 2\,\delta];H_0^1(\Omega)\times H_0(\curl;\Omega))$  of the problem
\begin{align*}
\mathcal{W}_1((\widetilde{u}_{,tt}(t),\widetilde{P}_{,tt}(t),(\varphi
,\Phi))=  {-}\mathcal{W}_2((\widetilde{u}(t),\widetilde{P}(t),(\varphi
,\Phi))+ l^{(f(t),M(t)}(\varphi,\Phi)
\end{align*}
for all $(\varphi, \Phi)\in H_0^1(\Omega)\times H_0(\curl;\Omega)$  and for all $t\in [\delta,2\,\delta]$,
which satisfies
\begin{align*}
&\widetilde{u}({x},\delta)={u}(
x, \delta),\quad \quad \quad\widetilde{u}_{,t}({x},\delta)={u}_{,t}(
x,\delta),\quad \quad \quad
\widetilde{P}({x},\delta)={P}(
x, \delta), \quad \quad \quad\widetilde{P}_{,t}({x},\delta)={P}_{,t}(
x, \delta),\ \ \text{\ \ }{x}\in {\Omega}.
\end{align*}%
Putting together the solution on $[0,\delta]$ and on $[\delta, 2\delta]$ we obtain a solution $(u,P)\in C^2([0,2\, \delta];$ $H_0^1(\Omega)\times H_0({\rm Curl};\Omega))$ of the initial problem  of the problem defined by \eqref{wformulation} and \eqref{icw}, due to the fact that $(u_{,tt}, P_{,tt})\in C([0,2\, \delta];H_0^1(\Omega)\times H_0({\rm Curl};\Omega))$.  {The continuity of $(u_{,tt}, P_{,tt})$ follows from the continuity on $[0,\delta]$ and $[\delta,2\, \delta]$, respectively, of the functions  $(u^*,P^*)$ defining the function $(u,P)$ through expressions \eqref{ustar}.}

Hence, due to similar iterations, we may extend the already constructed solution to the interval $[0,T]$ by considering a large enough step   {$n>2\, T\sqrt{c}$}, where $c$ is the constant from inequality \eqref{constc}.
\end{proof}

\begin{remark}
	\begin{itemize}
		\item[]
		\item[i)]   {We note that the application of the Banach fixed-point theorem on small intervals and to glue the fixed points together may be avoided by using an exponential weight in the maximum norm, as in the classical proof of the Picard-Lindel\"of theorem.  For a similar observation in this context and for other remarks regarding the presence of time derivatives in models one may consult \cite{Picardcerinta}.}
		\item[ii)]While in proving the existence of solution for the relaxed model without novel inertia terms (see \cite{GhibaNeffExistence,NeffGhibaMadeoLazar}) the requirement that  $\mathbb{C}_{\rm micro}$ is positive definite  was essential, in the  relaxed model which includes new  inertia terms, the existence result is still valid  when  $\widetilde{\mathbb{C}}_{\rm micro}$ is only positive semi-definite and  ${\mathbb{C}}_{\rm micro}$ is only bounded.
	\end{itemize}

\end{remark}

 \section{On the existence for a simplified model}
 \label{withoutPt}\setcounter{equation}{0}

 	The  aim of this   {section} is to investigate if the problem remains well-posed  when we take into account the following simplified expression for the   kinetic energy 
 	\begin{align*}
 	K\left(\u_{,t},\P_{,t},\nabla\u_{,t},\curl\,P_{,t}\right) & =\frac{1}{2}\left\langle \rho\,u_{,t},u_{,t}\right\rangle +\frac{1}{2}\left\langle \widetilde{\mathbb{C}}_{e}\,\sym\left(\grad\u_{,t}-P_{,t}\right),\sym\left(\grad\u_{,t}-P_{,t}\right)\right\rangle \\
 	&\quad +\frac{1}{2}  {\langle \widetilde{\mathbb{C}}_{\rm c}. \skew(\nabla u_{,t}-P_{,t}), \skew(\nabla u_{,t}-P_{,t})\rangle}+\langle \widetilde{\mathbb{C}}_{\rm micro}\,\sym\,P_{,t},\sym\,P_{,t}\rangle\\
 	&\quad  +\mu\:\frac{L_{c}^{2}}{2}\left\langle \widetilde{\mathbb{L}}_{\textrm{aniso}}\,\curl\,P_{,t},\curl\,P_{,t}\right\rangle 
 	\end{align*}
 and to the same potential energy density as in the previous subsection. A justification of a such choice is that in the limit case $P=\nabla u$, the variational formulation should be related   to the equations of the linear theory of nonlocal elasticity introduced by Eringen \cite{eringen1972linear} to fit the acoustical branch of elastic waves within the Brillouin zone in periodic one dimensional lattices \cite{brillouin2003wave}. We expect that the generalized model given in this section will improve the fitting of the dispersion  curves, at least up to a value of  the wave number compatible with the size of the microstructure of a considered microstructured material. 
 
    The equations of this model are 
 \begin{align}
 \rho\,u_{,tt}&-\dvg[ \widetilde{\C}_e. \sym(\nabla u_{,tt}-P_{,tt})+\widetilde{\C}_{\rm c}.\,{\rm skew}(\nabla u_{,tt}-P_{,tt})]=\notag\\&\ \ \ \ \dvg[ \mathbb{C}_{\rm e}. \sym(\nabla u-P)+{\C}_{\rm c}.\,{\rm skew}(\nabla u-P)]+f\, , \quad\qquad \qquad\qquad \qquad \qquad\ \ \quad \quad  
 \\
 {\mu}\,L_{c}^{2}\,\crl[ \widetilde{\mathbb{L}}_{\rm aniso}&.\crl\, P_{,tt}]-\widetilde{\C}_e. \sym (\nabla u_{,tt}-P_{,tt})-\widetilde{\C}_{\rm c}.\,{\rm skew}(\nabla u_{,tt}-P_{,tt})+\widetilde{\C}_{\rm micro}. \sym P_{,tt}=\notag\\&\quad - {\mu}\,L_{c}^{2}\,\crl[ {\mathbb{L}}_{\rm aniso}.\crl\, P]+\mathbb{C}_{\rm e}. \sym (\nabla u-P)+{\C}_{\rm c}.\,{\rm skew}(\nabla u-P)-{\C}_{\rm micro}. \sym P+M. \notag
 \end{align}
 
   {We assume that the remaining  constants of the model satisfy the conditions from the above section.}
 
  {In these hypotheses,  the corresponding bilinear form defined by the right-hand side of the system of partial differential equations, i.e.
 $\mathcal{W}_1:({H}_0^1(\Omega)
 \times{H}_0(\Curl;\Omega))\times ({H}_0^1(\Omega)
 \times{H}_0(\Curl;\Omega))\rightarrow\mathbb{R}$
 \begin{align}
 \mathcal{W}_1((u,P),(\varphi,\Phi))&=\dd\int _\Omega\biggl(\rho\,\langle u,\varphi\rangle+\langle \widetilde{\mathbb{C}}_{\rm e}. \sym(\nabla u-P), \sym(\nabla \varphi-\Phi)\rangle+\langle \widetilde{\mathbb{C}}_{\rm c}. \skew(\nabla u-P), \skew(\nabla \varphi-\Phi)\rangle\,\notag\\&\qquad \quad +\langle \widetilde{\mathbb{C}}_{\rm micro}. \sym P, \sym \Phi\rangle +\mu\, L_c^2\,\langle \widetilde{\mathbb{L}}_{\rm aniso}. \Curl\, P, \Curl\, \Phi\rangle\biggl)dv
\end{align}
remains bounded. Regarding its coercivity we point out that we have to impose, beside the conditions imposed in the last section, that 
  $\widetilde{\mathbb{C}}_{\rm micro}$ is positive definite.}

  {Under this additional assumption,   using the properties of the other constitutive tensors, we obtain that there is a $c>0$ such that
\begin{align*}
\mathcal{W}_1({w},{w})
&\geq c\,\dd\int _\Omega\biggl(\| u\|^2+\| \sym (\nabla u-P)\|^2+\|\sym P\|^2+\|\Curl\, P\|^2\biggl)\, dv
\end{align*}
for all
$  {w=(u,P)}\in H_0^1(\Omega)\times H_0(\curl;\Omega)$, which means that there is a $c>0$ such that
\begin{align*}
\mathcal{W}_1({w},{w})
&\geq c\,\dd\int _\Omega\biggl(\| \sym \nabla u\|^2+\|\sym P\|^2+\|\Curl\, P\|^2\biggl)\, dv
\end{align*}
for all
$  {w=(u,P)}\in H_0^1(\Omega)\times H_0(\curl;\Omega)$.}

Let us recall the following result  
\cite{NeffPaulyWitsch,NPW2,NPW3,BNPS2}:

\begin{theorem}\label{wdn}There exists a positive constant $C$, only depending on $\Omega$, such that for all $P\in{\rm H}_0({\rm Curl}\, ; \Omega)$ the following estimates hold:
	\begin{align*}
	{\| P\|_{H(\mathrm{Curl})}^2}:=\| P\|_{L^2(\Omega)}^{ {2}}+\| \Curl P\|_{L^2(\Omega)}^{ {2}}&\leq C\,(\| {\rm sym} P\|^2_{L^2(\Omega)}+\| \Curl P\|^2_{L^2(\Omega)}).
	\end{align*}
\end{theorem}
While in the model introduced in the previous section, we have used only the Poincar\'e inequality and Korn inequality to show the coercivity of the bilinear form from the left-hand side of the variational form of the boundary initial-value problem, in the model proposed in this section the  estimate in Theorem \ref{wdn} is essential. Indeed, using also the Korn inequality, we obtain that  $\mathcal{W}_1$  is coercive  for this model, too.
Therefore,   {under the additional assumption that $\widetilde{\mathbb{C}}_{\rm micro}$ is positive definite,} a similar analysis as in  the previous section shows that the model presented in this section is well-posed and the solution $(u,P)$ belongs to $ C^2([0,T];H_0^1(\Omega)\times H_0({\rm Curl};\Omega))$.

    { The positive definiteness of  ${\mathbb{C}}_{\rm e}$, ${\mathbb{C}}_{\rm c}$, ${\mathbb{C}}_{\rm micro}$,  ${\mathbb{L}}_{\rm aniso}$ or $\widetilde{\mathbb{C}}_{\rm c}$ is still not necessary in order to have a well-posed model. }

 \section{Final remarks}
 
 An approach based on the semigroup of linear operators may not lead  to an existence result, even of a weak solution, without an a priori assumption on the compatibility of the domains of the operators  defined by the left and right-hand side, respectively, of the system of partial differential equations (which are Banach spaces endowed with the corresponding graph-norms, since the operators generate $C_0$-contractive semigroups in $L^2(\Omega)\times L^2(\Omega)$).  We expect to have the same difficulties when the Galerkin method is used.  {However, for a model considering that the elastic tensors ${\mathbb{C}}_{\rm e}$, ${\mathbb{C}}_{\rm c}$, ${\mathbb{C}}_{\rm micro}$ and  ${\mathbb{L}}_{\rm aniso}$ vanish, the existence follows using techniques introduced in \cite{Picard1,Picard2}. In fact, in this particular case, the existence is based on the fact that the operator acting on $(u_{,tt},P_{,tt})$ is invertible, since it is of the form ${\rm Id}-\mathcal{A}$, where $\mathcal{A}$  generates a  $C_0$-contractive semigroup
 	in an appropiate Hilbert space, see \cite{GhibaNeffExistence}.}
 
    {Further, assuming  that $\widetilde{\mathbb{C}}_{\rm micro}$ is positive definite,  the model remains well-posed also in the quasistatic case, i.e the situation when   $\|\,P_{,t}\|$  and $\|\,u_{,t}\|$ are not present in the expression of the kinetic energy. }
  
 Another remark is that the first model considered in this paper remains well-posed also when the characteristic length scale $L_c$   {tends to zero}, i.e. $\Curl P_{,t}$ and $\Curl P$ are not present in the kinetic energy and in the potential energy density, respectively. The solution will belong to $ C^2([0,T];H_0^1(\Omega)\times L^2(\Omega))$  as long as the forces are in $ C([0,T];H^{-1}(\Omega)\times L^2(\Omega))$  and the initial conditions are assumed to be in $H_0^1(\Omega)\times L^2(\Omega)$ (see also \cite{Hlavacek75} for  a related model). Note also that for   {$L_c\to0$} we do not  prescribe $P$ on the boundary. 
 
 We are not able to say the same for the second model, in the case   {$L_c\to0$}, since the bilinear form $\mathcal{W}_1$ may not be coercive, in general. However, when $\mathbb{C}_{\rm c}=0$ and $\widetilde{\mathbb{C}}_{\rm c}=0$,   {$L_c\to0$} and $\|P_{,t}\|$ is not taken into account in the form of the kinetic energy, the problem involves actually only the functions $u$ and $\sym\, P$. Therefore, in this situation and when $\widetilde{\mathbb{C}}_{\rm micro}$ is positive definite, if $\sym\, P(0), \sym\,{P}_{,t}(0)\in  L^2(\Omega)$, $u(0), {u}_{,t}(0)\in H_0^1(\Omega)$, $u=0$ on the boundary, $P$ is also prescribed on the boundary, and the  forces are in $ C([0,T];H^{-1}(\Omega)\times L^2(\Omega))$,  there exists a unique   solution $(u,\sym \,P)\in C^2([0,T];H_0^1(\Omega)\times L^2(\Omega))$ of the problem in $u$ and $\sym\, P$.  We do not have information about $\skew \,P$, since it is not involved in the equations and clearly it may  not be uniquely determined, in order to satisfy also the initial and the boundary conditions for the full $P$.  
  
  \section{Acknowledgement}   We are sincerely grateful to the reviewers for their insightful comments which helped improve the paper.
  The work of I.D. Ghiba was supported by Alexandru Ioan Cuza University of Iasi (UAIC)
  under the grant GI-17458, within the internal grant competition for
  young researchers.

\begin{footnotesize}
\bibliographystyle{plain} 

\end{footnotesize}

\end{document}

%% file: neffstyle.tex
\newcommand{\ds}{\,{\rm{ds}} }

\DeclareMathOperator{\curl}{curl}

\DeclareMathOperator{\Curl}{Curl}

\newcommand{\Co}{C_0^{\infty}(\Omega)}

\newcommand{\norm}[1]{\|#1\|}

\newcommand{\R}{\mathbb{R}}

\newcommand{\C}{\mathbb{C}}

\newcommand{\D}{\mathbb{D}}

\renewcommand{\skew}{\mathop{\rm skew}}

\DeclareMathOperator{\sym}{sym}

\DeclareMathOperator{\axl}{axl}

\DeclareMathOperator{\dev}{dev}

\DeclareMathOperator{\so}{\mathfrak{so}}
\DeclareMathOperator{\gl}{\mathfrak{gl}}

\DeclareMathOperator{\Lin}{Lin}

\newcommand{\Sym}{ {\rm{Sym}} }

\newcommand{\Mprod}[2]{ {\langle #1 ,#2\rangle} }

\newcommand{\tr}[1]{ {\Tr \left[{#1}\right]} }

%% file: NSRevision.bbl
\begin{thebibliography}{10}
	
	\bibitem{aivaliotis2018low}
		A. Aivaliotis, A. Daouadji, G.  Barbagallo, D. Tallarico, P. Neff,  and A. Madeo. 	
		\newblock {\em Low-and high-frequency Stoneley waves, reflection and transmission at a Cauchy/relaxed micromorphic interface},
			\newblock {arXiv preprint arXiv:1810.12578},
		{2018}
	
	\bibitem{Adams75}
	R.A. Adams.
	\newblock {\em Sobolev {S}paces.}, volume~65 of {\em Pure and {A}pplied
		{M}athematics}.
	\newblock Academic Press, London, 1. edition, 1975.
	
	\bibitem{askes2011increasing}
	H.~Askes, D.~Nguyen, and A.~Tyas.
	\newblock Increasing the critical time step: micro-inertia, inertia penalties
	and mass scaling.
	\newblock {\em Comp. Mech.}, 47(6):657--667, 2011.
	
	\bibitem{barbagallo2017transparent}
	G.~Barbagallo, A.~Madeo, M.V. d'Agostino, R.~Abreu, I.D. Ghiba, and P.~Neff.
	\newblock Transparent anisotropy for the relaxed micromorphic model:
	macroscopic consistency conditions and long wave length asymptotics.
	\newblock {\em Int. J. Solids Struct.}, 120:7--30, 2017.
	
	\bibitem{Barbagalo2018} G. Barbagallo, D. Tallarico, M.V. D'Agostino, A. Aivaliotis, P. Neff, A.  Madeo. 
	\newblock Relaxed micromorphic model of transient wave propagation in anisotropic band-gap metastructures. \newblock {\em Int. J. Solids Struct.}, https://doi.org/10.1016/j.ijsolstr.2018.11.033, 2018.
	
	
	\bibitem{BNPS2}
	S.~Bauer, P.~Neff, D.~Pauly, and G.~Starke.
	\newblock Dev-{Div} and {DevSym}-{DevCurl} inequalities for incompatible square
	tensor fields with mixed boundary conditions.
	\newblock {\em ESAIM: COCV}, 22(1):112--133, 2016.
	
	\bibitem{bennett2007elasticity}
	T.~Bennett, I.M. Gitman, and H.~Askes.
	\newblock Elasticity theories with higher-order gradients of inertia and
	stiffness for the modelling of wave dispersion in laminates.
	\newblock {\em Int. J. Fract.}, 148(2):185--193, 2007.
	
	
	
	\bibitem{brillouin2003wave}
	L.~Brillouin.
	\newblock {\em Wave propagation in periodic structures: electric filters and
		crystal lattices}.
	\newblock Courier Corporation, 2003.
	
	\bibitem{chelminski2016renormalized}
	K.~Che{\l}mi{\'n}ski and S.~Owczarek.
	\newblock Renormalized solutions in thermo-visco-plasticity for a
	{Norton--Hoff} type model. {Part I}: The truncated case.
	\newblock {\em Nonlinear Anal. Real World Appl.}, 28:140--152, 2016.
	
	\bibitem{chen2001dispersive}
	W.~Chen and J.~Fish.
	\newblock A dispersive model for wave propagation in periodic heterogeneous
	media based on homogenization with multiple spatial and temporal scales.
	\newblock {\em J. Applied Mech.}, 68(2):153--161, 2001.
	
	\bibitem{Eringen_Claus69}
	W.D. Claus and A.C. Eringen.
	\newblock Three dislocation concepts and micromorphic mechanics.
	\newblock In {\em Developments in {M}echanics, {Proceedings} of the 12th
		{Midwestern Mechanics Conference}}, volume~6, pages 349--358. Midwestern,
	1969.
	
	\bibitem{Cosserat09}
	E.~Cosserat and F.~Cosserat.
	\newblock {\em Th\'eorie des corps d\'eformables.}
	\newblock Librairie Scientifique A. Hermann et Fils (engl. translation by D.
	Delphenich 2007, pdf available at
	http://www.mathematik.tu-darmstadt.de/fbereiche/analysis/pde/staff/neff/patrizio/{C}osserat.html),
	reprint 2009 by Hermann Librairie Scientifique, ISBN 978 27056 6920 1, Paris,
	1909.
	
	\bibitem{d2017effective}
	M.V. d'Agostino, G.~Barbagallo, I.D. Ghiba, B.~Eidel, P.~Neff, and A.~Madeo.
	\newblock Effective description of anisotropic wave dispersion in mechanical
	meta-materials via the relaxed micromorphic model.
	\newblock {\em arXiv preprint arXiv:1709.07054}, 2018.
	
	\bibitem{d2017panorama}
	M.V. d'Agostino, G.~Barbagallo, I.D. Ghiba, A.~Madeo, and P.~Neff.
	\newblock A panorama of dispersion curves for the weighted isotropic relaxed
	micromorphic model.
	\newblock {\em Z. Angew. Math. Mech.}, 97(11):1436--1481, 2017.
	
	\bibitem{domenico2016new}
	D.~De Domenico and H.~Askes.
	\newblock A new multi-scale dispersive gradient elasticity model with
	micro-inertia: Formulation and-finite element implementation.
	\newblock {\em Int. J. Numer. Meth. in Eng.}, 108(5):485--512, 2016.
	
	\bibitem{eringen1972linear}
	A.C. Eringen.
	\newblock Linear theory of nonlocal elasticity and dispersion of plane waves.
	\newblock {\em Int. J. Eng. Sci.}, 10(5):425--435, 1972.
	
	\bibitem{Eringen99}
	A.C. Eringen.
	\newblock {\em Microcontinuum {F}ield {T}heories.}
	\newblock Springer, Heidelberg, 1999.
	
	\bibitem{eringen1972nonlocal}
	A.C. Eringen and D.G.B. Edelen.
	\newblock On nonlocal elasticity.
	\newblock {\em Int. J. Eng. Sci.}, 10(3):233--248, 1972.
	
	\bibitem{Eringen64}
	A.C. Eringen and E.S. Suhubi.
	\newblock Nonlinear theory of simple micro-elastic solids.
	\newblock {\em Int. J. Eng. Sci.}, 2:189--203, 1964.
	
	\bibitem{blanco2000large}
	A.~Blanco,  E. Chomski, S.Grabtchaket,  M. Ibisate, S.  John~ et al.,
	\newblock Large-scale synthesis of a silicon photonic crystal with a complete
	three-dimensional bandgap near 1.5 micrometres.
	\newblock {\em Nature}, 405(6785):437--440, 2000.
	
	\bibitem{liu2000locally}
	Z.~Liu, Z.  Xixiang, M. Yiwei, Y. Y. Zhu, Y. Zhiyu, C. T. Chan, and S. Ping,
	\newblock Locally resonant sonic materials.
	\newblock {\em Science}, 289(5485):1734--1736, 2000.
	
	\bibitem{Forest06}
	S.~Forest and R.~Sievert.
	\newblock Nonlinear microstrain theories.
	\newblock {\em Int. J. Solids Struct.}, 43:7224--7245, 2006.
	
	\bibitem{germain1973method}
	P.~Germain.
	\newblock The method of virtual power in continuum mechanics. {Part 2:
		Microstructure}.
	\newblock {\em SIAM J. Appl. Math.}, 25(3):556--575, 1973.
	
	\bibitem{GhibaNeffExistence}
	I.~D. Ghiba, P.~Neff, A.~Madeo, L.~Placidi, and G.~Rosi.
	\newblock The relaxed linear micromorphic continuum: Existence, uniqueness and
	continuous dependence in dynamics.
	\newblock {\em Math. Mech. Solids}, 20:1171--1197, 2015.
	
	\bibitem{Raviart79}
	V.~Girault and P.A. Raviart.
	\newblock {\em Finite Element Approximation of the {N}avier-{S}tokes
		Equations.}, volume 749 of {\em Lect. Notes Math.}
	\newblock Springer, Heidelberg, 1979.
	
	\bibitem{gitman2007gradient}
	I.M. Gitman, H.~Askes, and E.C. Aifantis.
	\newblock Gradient elasticity with internal length and internal inertia based
	on the homogenisation of a representative volume element.
	\newblock {\em J. Mech. Behav. Biomed. Mater.}, 18(1):1--16, 2007.
	
	\bibitem{Hlavacek69}
	I.~Hlav{\'a}{\v c}ek and M.~Hlav{\'a}{\v c}ek.
	\newblock On the existence and uniqueness of solutions and some variational
	principles in linear theories of elasticity with couple-stresses. {I}:
	{C}osserat continuum. {II}: {M}indlin's elasticity with micro-structure and
	the first strain gradient.
	\newblock {\em J. Apl. Mat.}, 14:387--426, 1969.
	
	\bibitem{Hlavacek75}M.~Hlav{\'a}{\v c}ek.
	\newblock A continuum theory for isotropic two-phase elastic composites. 
	\newblock {\em Int. J. Solids Struct.}, 11:1137-1144, 1975.
	
	
	\bibitem{IesanNappa2001}
	D.~Ie\c{s}an.
	\newblock Extremum principle and existence results in micromorphic elasticity.
	\newblock {\em Int. J. Eng. Sci.}, 39:2051--2070, 2001.
	
	\bibitem{Iesan2002}
	D.~Ie\c{s}an.
	\newblock On the micromorphic thermoelasticity.
	\newblock {\em Int. J. Eng. Sci.}, 40:549--567, 2002.
	

	
	
	\bibitem{Neff_JeongMMS08}
	J.~Jeong and P.~Neff.
	\newblock Existence, uniqueness and stability in linear {C}osserat elasticity
	for weakest curvature conditions.
	\newblock {\em Math. Mech. Solids}, 15(1):78--95, 2010.
	
	\bibitem{Leis86}
	R.~Leis.
	\newblock {\em Initial Boundary Value problems in Mathematical Physics}.
	\newblock Teubner, Stuttgart, 1986.
	
	\bibitem{madeo2018modeling}
	A.~Madeo, M.~Collet, M.~Miniaci, K.~Billon, Morvan M.~Ouisse, and P.~Neff.
	\newblock Modeling phononic crystals via the weighted relaxed micromorphic
	model with free and gradient micro-inertia.
	\newblock {\em J. Elasticity}, 130(1):59--83, 2018.
	
	\bibitem{madeo2017role}
	A.~Madeo, P.~Neff, E.~Aifantis, G.~Barbagallo, and M.V. d'Agostino.
	\newblock On the role of micro-inertia in enriched continuum mechanics.
	\newblock {\em Proc. R. Soc. A}, 473(2198):20160722, 2017.
	
	\bibitem{madeo2017review}
	A.~Madeo, P.~Neff, G.~Barbagallo, M.V. d'Agostino, and I.D. Ghiba.
	\newblock A review on wave propagation modeling in band-gap meta-materials via
	enriched continuum models.
	\newblock In {\em Mathematical Modelling in Solid Mechanics}, pages 89--105.
	Springer, 2017.
	
	\bibitem{madeo2016complete}
	A.~Madeo, P.~Neff, M.V. d'Agostino, and G.~Barbagallo.
	\newblock Complete band gaps including non-local effects occur only in the
	relaxed micromorphic model.
	\newblock {\em Comptes Rendus M{\'e}canique}, 344(11-12):784--796, 2016.
	
	\bibitem{MadeoNeffGhibaW}
	A.~Madeo, P.~Neff, I.~D. Ghiba, L.~Placidi, and G.~Rosi.
	\newblock Wave propagation in relaxed linear micromorphic continua: modelling
	meta-materials with frequency band-gaps.
	\newblock {\em Cont. Mech. Therm.}, 27:551--570, 2015.
	
	\bibitem{madeo2016reflection}
	A.~Madeo, P.~Neff, I.D. Ghiba, and G.~Rosi.
	\newblock Reflection and transmission of elastic waves in non-local band-gap
	meta-materials: a comprehensive study via the relaxed micromorphic model.
	\newblock {\em J. Mech. Physics of Solids}, 95:441--479, 2016.
	
	\bibitem{Mindlin64}
	R.D. Mindlin.
	\newblock Micro-structure in linear elasticity.
	\newblock {\em Arch. Rat. Mech. Anal.}, 16:51--77, 1964.
	
	\bibitem{Neff00b}
	 {P. Neff},
		\newblock {On {K}orn's first inequality with nonconstant coefficients.},
		\newblock {\em Proc. Roy. Soc. Edinb. A},
		{132}:{221-243},
		{2002} 
		
	\bibitem{Neff_micromorphic_rse_05}
	P.~Neff.
	\newblock Existence of minimizers for a finite-strain micromorphic elastic
	solid.
	\newblock {\em Proc. Roy. Soc. Edinb. A}, 136:997--1012, 2006.
	
	\bibitem{neff2015existence}
	P.~Neff, M.~B{\^\i}rsan, and F.~Osterbrink.
	\newblock Existence theorem for geometrically nonlinear cosserat micropolar
	model under uniform convexity requirements.
	\newblock {\em J. Elasticity}, 121(1):119--141, 2015.
	
	\bibitem{Neff_Chelminski05_dyn}
	P.~Neff and K.~Che{\l}mi\'nski.
	\newblock Well-posedness of dynamic {C}osserat plasticity.
	\newblock {\em Appl. Math. Optim.}, 56:19--35, 2007.
	
	\bibitem{Neff_Forest07}
	P.~Neff and S.~Forest.
	\newblock A geometrically exact micromorphic model for elastic metallic foams
	accounting for affine microstructure. {Modelling}, existence of minimizers,
	identification of moduli and computational results.
	\newblock {\em J. Elasticity}, 87:239--276, 2007.
	
	\bibitem{NeffGhibaMadeoLazar}
	P.~Neff, I.~D. Ghiba, M.~Lazar, and A.~Madeo.
	\newblock The relaxed linear micromorphic continuum: well-posedness of the
	static problem and relations to the gauge theory of dislocations.
	\newblock {\em Q. J. Mech. Appl. Math.}, 68:53--84, 2015.
	
	\bibitem{NeffGhibaMicroModel}
	P.~Neff, I.~D. Ghiba, A.~Madeo, L.~Placidi, and G.~Rosi.
	\newblock A unifying perspective: the relaxed linear micromorphic continuum.
	\newblock {\em Cont. Mech. Therm.}, 26:639--681, 2014.
	
	\bibitem{neff2017real}
	P.~Neff, A.~Madeo, G.~Barbagallo, M.V. d'Agostino, R.~Abreu, and I.D. Ghiba.
	\newblock Real wave propagation in the isotropic-relaxed micromorphic model.
	\newblock {\em Proc. R. Soc. A}, 473(2197):20160790, 2017.
	
	\bibitem{Neff_curl06}
	P.~Neff and I.~M\"unch.
	\newblock Curl bounds {Grad} on {${\rm SO}(3)$}.
	\newblock {\em ESAIM: Control, Optimisation and Calculus of Variations},
	14(1):148--159, 2008.
	
	\bibitem{NPW2}
	P.~Neff, D.~Pauly, and K.J. Witsch.
	\newblock A canonical extension of {K}orn's first inequality to {${\rm
			H(Curl)}$} motivated by gradient plasticity with plastic spin.
	\newblock {\em C. R. Acad. Sci. Paris, Ser. I}, 349:1251--1254, 2011.
	
	\bibitem{NPW3}
	P.~Neff, D.~Pauly, and K.J. Witsch.
	\newblock Maxwell meets {K}orn: a new coercive inequality for tensor fields in
	{$\mathbb{R}^{N\times N}$} with square-integrable exterior derivative.
	\newblock {\em Math. Methods Appl. Sci.}, 35:65--71, 2012.
	
	\bibitem{NeffPaulyWitsch}
	P.~Neff, D.~Pauly, and K.J. Witsch.
	\newblock Poincar\'{e} meets {K}orn via {M}axwell: Extending {K}orn's first
	inequality to incompatible tensor fields.
	\newblock {\em J. Differential Equations}, 258:1267--1302, 2015.
	
	\bibitem{NunziatoCowin79}
	J.W. Nunziato and S.C. Cowin.
	\newblock A nonlinear theory of elastic materials with voids.
	\newblock {\em Arch. Rat. Mech. Anal.}, 72:175--201, 1979.
	
	\bibitem{Picard1} R.~Picard.
	\newblock A structural observation for linear material laws in classical mathematical physics. 
	\newblock {\em Math. Meth.  Appl. Sci.}, 32:1768-1803, 2009.
	
\bibitem{Picard2} R.~Picard.
\newblock Mother operators and their descendants.
\newblock {\em J. Math. Anal.  Appl.}, 403: 54-62, 2013.
	
	
	\bibitem{Picardcerinta} R.~Picard, S.~Trostorff, and M.~Waurick.
	\newblock A functional analytic perspective to delay differential equations. 
	\newblock {\em Operators and Matrices}, 8 (1): 217-236, 2014.
	
	\bibitem{picard}
	R.~Picard, S.~Trostorff, and M.~Waurick.
	\newblock On some models for elastic solids with micro-structure.
	\newblock {\em Z. Angew. Math. Mech.}, 95(7):664--689, 2014.
	
	\bibitem{rubin1995continuum}
	M.B. Rubin, Ph. Rosenau, and O.~Gottlieb.
	\newblock Continuum model of dispersion caused by an inherent material
	characteristic length.
	\newblock {\em J. Appl. Phys.}, 77(8):4054--4063, 1995.
	
	\bibitem{Soos69}
	E.~So\'{o}s.
	\newblock Uniqueness theorems for homogeneous, isotropic, simple elastic and
	thermoelasticmaterials having a microstructure.
	\newblock {\em Int. J. Eng. Sci.}, 7:257--268, 1969.
	
	\bibitem{Teodosiu}
	C.~Teodosiu.
	\newblock Discussion on {Papers} by {A.C. Eringen} and {W.D. Claus, Jr.}, and
	{N. Fox}.
	\newblock In J.A. Simmons, R.~de~Wit, and R.~Bullough, editors, {\em
		Fundamental Aspects of Dislocation Theory.}, volume~1 of {\em Nat. Bur.
		Stand. (U.S.), Spec. Publ.}, pages 1054--1059. Spec. Publ., 1970.
	
	\bibitem{wang2002modeling}
	Z.P. Wang and C.T. Sun.
	\newblock Modeling micro-inertia in heterogeneous materials under dynamic
	loading.
	\newblock {\em Wave Motion}, 36(4):473--485, 2002.
	
\end{thebibliography}
